\documentclass[11pt]{article}

\usepackage{epsfig}
\usepackage[a4paper, top=3cm, bottom=3cm, left=1.75cm, right=1.75cm]{geometry}
\usepackage{amsfonts,amsmath}
\usepackage{graphicx}
\usepackage{amssymb}
\usepackage{hyperref}
\usepackage{url}
\usepackage{booktabs}
\usepackage{amsfonts}
\usepackage{nicefrac}
\usepackage{microtype}
\usepackage{pifont}
\usepackage{cases}
\usepackage{graphicx}
\usepackage{amsmath}
\usepackage{doi}
\usepackage{amsthm}
\usepackage{dsfont}
\usepackage{mathtools}
\usepackage{tcolorbox}
\usepackage{amsmath}
\usepackage{physics}
\newcommand{\eremk}{\hbox{}\hfill\rule{0.8ex}{0.8ex}}
\usepackage{bbm}
\usepackage{wasysym}
\usepackage{caption}
\usepackage{subcaption}
\usepackage{enumerate}
\usepackage{accents,scalerel}

\definecolor{myblue}{HTML}{1f77b4}
\definecolor{myorange}{HTML}{ff4c0e}
\definecolor{myyellow}{HTML}{bcbd22}
\definecolor{mypurple}{HTML}{9467bd}
\definecolor{mygreen}{HTML}{2ca02c}

\newtheorem{theorem}{Theorem}[section]
\newtheorem{assumption}[theorem]{Assumption}
\newtheorem{lemma}[theorem]{Lemma}
\newtheorem{proposition}[theorem]{Proposition}

\newtheorem{remark}[theorem]{Remark}

\newcommand{\bx}{\boldsymbol{x}}

\newcommand{\bh}{\boldsymbol{h}}
\renewcommand{\div}{\text{div}}

\DeclareFontFamily{U}{matha}{\hyphenchar\font45}
\DeclareFontShape{U}{matha}{m}{n}{
<-6> matha5 <6-7> matha6 <7-8> matha7
<8-9> matha8 <9-10> matha9
<10-12> matha10 <12-> matha12
}{}
\DeclareSymbolFont{matha}{U}{matha}{m}{n}
\DeclareFontFamily{U}{mathx}{\hyphenchar\font45}
\DeclareFontShape{U}{mathx}{m}{n}{
<-6> mathx5 <6-7> mathx6 <7-8> mathx7
<8-9> mathx8 <9-10> mathx9
<10-12> mathx10 <12-> mathx12
}{}
\DeclareSymbolFont{mathx}{U}{mathx}{m}{n}
\DeclareMathDelimiter{\vvvert} {0}{matha}{"7E}{mathx}{"17}%

\DeclarePairedDelimiterX{\normiii}[1]
{\vvvert}
{\vvvert}
{\ifblank{#1}{\:\cdot\:}{#1}}

\renewcommand{\phi}{\varphi}

\newcommand{\R}{\mathbb{R}}

\newcommand{\N}{\mathbb{N}}

\newcommand{\Int}{\mathcal{I}_{h_t}^{p_t -1}}

\newcommand{\USG}{{\stackrel{*}{U}}_{\bh}^{\raisebox{-1ex}{$\scriptstyle p_t + 1$}}}

\newcommand{\NL}{g}

\newcommand{\GLe}{\mathsf{GLe}}
\newcommand{\GLo}{\mathsf{GLo}}

\newcommand{\mdot}[1]{\accentset{\scaleobj{0.6}{\bullet}}{#1}}

\usepackage{pifont}

\title{Stability, convergence, and geometric properties of second-order-in-time space--time discretizations for linear and semilinear wave equations}
\author{$^a$Matteo~Ferrari, $^a$Ilaria~Perugia, $^a$Enrico~Zampa \vspace{0.5cm}}

\date{
$^a$ Faculty of Mathematics, University of Vienna\\
Oskar-Morgenstern-Platz 1, 1090 Vienna, Austria}

\begin{document}
\maketitle

%%%%%%%%%%%%%%%%%%%%%%%%%%%%
\begin{abstract}
\noindent 
We revisit second-order-in-time space--time discretizations of the linear and semilinear wave equations by establishing precise equivalences with first-order-in-time formulations. Focusing on schemes using continuous piecewise-polynomial trial functions in time, we analyze their stability, convergence, and geometric properties. We consider first a weak space--time formulation with test functions projected onto discontinuous polynomials of one degree lower in time, showing that it is equivalent to the scheme proposed in [French, Peterson 1996] in the linear case, and extended in [Karakashian, Makridakis 2005] to the semilinear case. In particular, this equivalence shows that this method conserves energy at mesh nodes but is not symplectic. We then introduce two symplectic variants, obtained through Gauss--Legendre and Gauss--Lobatto quadratures in time, and show that they correspond to specific Runge--Kutta time integrators. These connections clarify the geometric structure of the space–time methods considered.
\end{abstract}
%%%%%%%%%%%%%%%%%%%%%%%%%%%%

\section{Introduction}

\noindent 
We study a class of second-order-in-time space--time Galerkin methods for the linear and semilinear wave equations and establish novel equivalence results with first-order-in-time formulations. These equivalences enable the rigorous transfer of known properties from first-order to second-order schemes.

\smallskip
\noindent
Second-order-in-time space--time Galerkin schemes for the wave equation have recently attracted significant interest. A common construction begins by multiplying the wave equation by test functions of space and time that vanish at the final time, followed by integration by parts in both space and time. The initial condition for the solution is then imposed strongly, while the initial condition for its time derivative is enforced weakly. The resulting formulation is then discretized with continuous finite elements in both space and time. For the linear wave equation, it is well known that such a formulation may fail to be unconditionally stable when standard discrete time spaces are used; see \cite[Remark 4.9]{SteinbachZank2020} for continuous piecewise linear polynomials, and \cite[Theorem 5.9]{FerrariFraschini2024} for maximal regularity splines. This instability arises from the lack of inf--sup stability in the energy norm \cite[Theorem 4.2.23]{Zank2020}. To overcome this issue, various stabilization techniques have been proposed.

\smallskip
\noindent
In \cite{SteinbachZank2019}, a stabilization technique was introduced for continuous piecewise linear polynomials in time, in which the test function in the spatial stiffness term is projected onto a space of piecewise constant functions in time. This approach was later extended and numerically validated for higher-order continuous piecewise polynomials in~\cite{Zank2021}, by projecting onto discontinuous piecewise polynomial of one degree lower. For continuous piecewise linear polynomials, the stabilization proposed in~\cite{SteinbachZank2019} is equivalent to adding a non-consistent penalty term. Along these lines, a stabilization method was proposed in~\cite{FraschiniLoliMoiolaSangalli2023} (and subsequently analyzed in~\cite{FerrariFraschini2024}) for maximal regularity splines in time. Another approach involves transforming the test functions using an operator: the modified Hilbert transform, which has yielded promising numerical results in~\cite{LoscherSteinbachZank2023}, or via Morawetz multiplier \cite{BignardiMoiola2023}.

\smallskip
\noindent
In this paper, we prove that a slight modification of the stabilized formulation of~\cite{Zank2021}
is equivalent to the first-order-in-time discontinuous Galerkin--continuous Galerkin (DG--CG) method introduced in~\cite{BalesLasiecka1994,FrenchPeterson1996} for the linear wave equation. This implies that the second-order-in-time method possesses the same theoretical properties that have been established for the first-order-in-time method in~\cite{FrenchPeterson1996,Gomez2025}. Since the modification affects only the right-hand side term, the stability results also carry over to the original formulation of~\cite{Zank2021}. 
This equivalence generalizes the result of~\cite{Zank2019}, which showed that on tensor product meshes the lowest-order discretization in~\cite{Zank2021} coincides with the Crank--Nicolson scheme in time (equivalently, Newmark with parameters~$(\beta,\gamma)=(1/4,1/2)$). Up to the treatment of the right-hand side, this scheme is itself equivalent to the lowest-order DG--CG method. We then apply the same framework to the semilinear wave equation. In all terms of the resulting bilinear form, except for the temporal stiffness term, a projection onto the space of discontinuous polynomials of one degree lower is inserted. The resulting formulation is equivalent to the extension of the first-order-in-time DG--CG method analyzed in~\cite{KarakashianMakridakis2005}, so that the stability and convergence established there carry over directly to the second-order-in-time scheme. While preserving the unconditional stability and convergences properties established in \cite{FrenchPeterson1996, KarakashianMakridakis1998, Gomez2025}, the second-order-in-time formulation reduces the number of unknowns by half. Moreover, although it uses trial and test functions that are continuous in time, it can still be implemented in a time-stepping fashion and efficient solvers are available, see \cite[\S 7]{Zank2025} (and also \cite{Tani2017,LangerZank2021,LoliSangalli2025}).

\smallskip
\noindent
The methods discussed so far exactly conserve a discrete energy. In practice, however, nonlinear terms are integrated using quadrature, so energy conservation holds only up to quadrature error. 
Quadrature can also be exploited to preserve other geometric properties, such as symplecticity. Building on this, we propose two symplectic variants based on Gauss--Legendre and Gauss--Lobatto quadratures. On tensor product meshes, these correspond, respectively, to applying the Gauss--Legendre Runge--Kutta or the Lobatto IIIA/IIIB Runge--Kutta time discretizations to the first-order-in-time equivalent formulation. As a result, they are symplectic, with the Gauss--Legendre variant unconditionally stable and the Gauss--Lobatto variant only conditionally stable. We summarize the main properties of the considered methods in Table~\ref{tab:properties}.

\begin{table}[h!]
\centering
\begin{tabular}{|l|c|c|c|}
\hline
\textbf{Method} & \textbf{Stability} & \textbf{Energy preservation} & \textbf{Symplecticity} \\
\hline
Unstabilized & $\triangle$  & $\times$ & $\checkmark$ \\
Stabilized & $\checkmark$ & $\checkmark$ (up to quadrature error) & $\times$ \\
Gauss--Legendre & $\checkmark$ & $\times$ & $\checkmark$ \\
Gauss--Lobatto  & $\triangle$  & $\times$ & $\checkmark$ \\
\hline
\end{tabular}
\caption{Properties of four second-order-in-time methods: ``Unstabilized'' is the basic method, ``Stabilized'' refers to the method with the projection onto discontinuous polynomials of one degree lower, and ``Gauss--Legendre'' and ``Gauss--Lobatto'' are the methods with the corresponding quadratures.
$\checkmark$ = yes, $\times$ = no, $\triangle$ = conditional.}
\label{tab:properties}
\end{table}

\medskip
\noindent 
{\bf Outline.} The paper is organized as follows. In Section~\ref{sec:2}, we focus on the second-order-in-time space--time method for the linear wave equation and establish its equivalence with the first-order-in-time formulation. This yields an unconditionally stable scheme that can be efficiently implemented using under-integration in time. In Section~\ref{sec:3}, we extend the same framework to the semilinear wave equation. Section~\ref{sec:33} presents symplectic variants based on Gauss--Legendre and Gauss--Lobatto quadratures, which preserve key geometric properties.
Numerical results reported in sections \ref{sec:251} and \ref{sec:321} validate, respectively, energy preservation in the linear case and the convergence rates for the sine-Gordon equation.

\section{The linear case} \label{sec:2}

In this section, we consider a second-order-in-time space--time method for the \emph{linear} wave equation, which is a slight modification of the stabilized approach from~\cite{Zank2021}. We prove an equivalence result between this formulation and the space–time Discontinuous–Continuous Galerkin (DG–CG) discretization of the problem in its first-order-in-time form, as introduced in~\cite{BalesLasiecka1994, FrenchPeterson1996}, allowing theoretical results established for the latter to be extended to the former. We also describe an efficient implementation of the second-order-in-time method that avoids solving local linear systems for computing the~$L^2$ projections required by the formulation.

\subsection{The linear wave equation}
We consider the following initial-value problem with homogeneous Dirichlet boundary conditions
\begin{equation} \label{eq:1}
	\begin{cases}
		\partial_{t}^2 U - \div_{\bx}(c^2 \nabla_{\bx} U)  =  F & \text{in~} Q_T := \Omega \times (0, T), 
		\\ U = 0 &  \text{on~} \partial \Omega \times (0,T),
		\\ U(\cdot,0) = U_0, \quad \partial_t U(\cdot,t)_{|_{t=0}} = V_0 & \text{in~} \Omega,
	\end{cases}
\end{equation}
where~$\Omega \subset \R^d$ ($d = 1, 2, 3$) is a 
polytopal Lipschitz domain,~$T > 0$ is a finite time, and the wave velocity $c \in C(\overline{\Omega})$, $c(\bx) \ge c_0 > 0$ for all $\bx \in \Omega$. Moreover, the given data satisfy
\begin{equation*}
     U_0 \in H^1_0(\Omega), \quad V_0 \in L^2(\Omega), \quad F \in L^1(0,T;L^2(\Omega)).
\end{equation*}
Consider the following variational formulation of problem~\eqref{eq:1}: find $U \in C([0,T];H_0^1(\Omega)) \cap C^1([0,T];L^2(\Omega))$ such that $U(\cdot,0) = U_0$ in $H_0^1(\Omega)$, $\partial_t U(\cdot,t)|_{t=0} = V_0$ in $L^2(\Omega)$, and 
\begin{equation} \label{eq:2}
    \int_0^T \langle \partial_t^2 U(\cdot,t), W(\cdot,t) \rangle_{H^1_0(\Omega)} \, \dd t + (c^2 \nabla_{\bx} U, \nabla_{\bx} W)_{L^2(Q_T)} = \int_0^T (F(\cdot,t),W(\cdot,t))_{L^2(\Omega)} \, \dd t
\end{equation}
for all $W \in L^2(0,T;H^1_0(\Omega))$. Here, $\langle \cdot, \cdot \rangle_{H^1_0(\Omega)}$ denotes the duality pairing between $[H_0^1(\Omega)]'$ and $H_0^1(\Omega)$. Existence and uniqueness of a solution to \eqref{eq:2} is shown in \cite[Theorem 2.1 in Part I]{LasieckaTriggiani1994}. Moreover, any solutions of~\eqref{eq:2} actually satisfies
\begin{equation} \label{eq:3}
    U \in C([0,T];H_0^1(\Omega)) \cap C^1([0,T];L^2(\Omega)) \cap H^2(0,T;[H_0^1(\Omega)]').
\end{equation}
With the auxiliary unknown~$V := \partial_t U$, problem~\eqref{eq:1} can also be formulated as the first-order-in-time system
\begin{equation} \label{eq:4}
	\begin{cases}
		\partial_t U - V=0 & \text{in } Q_T, 
		\\ \partial_t V - \text{div}_{\bx} (c^2 \nabla_{\bx}  U)  =  F & \text{in } Q_T, 
        \\ U = 0 & \text{on~} \partial \Omega \times (0,T),
		\\ U(\cdot,0) = U_0, \quad V(\cdot,0) = V_0 & \text{in~} \Omega.
	\end{cases}
\end{equation}
With $U$ as in \eqref{eq:3}, we have
\begin{equation*}
    V = \partial_t U \in C([0,T];L^2(\Omega)) \cap H^1(0,T;[H_0^1(\Omega)]').
\end{equation*}
Before formulating space--time tensor product discretizations of problem~\eqref{eq:1}, we introduce some notation.

\subsection{Discretization setup} \label{sec:22}
\noindent
 For the spatial discretization, let us consider a space~$S_{h_{\bx}}^{p_{\bx}}(\Omega) \subset H_0^1(\Omega)$ of piecewise polynomial functions of degree at most~$p_{\bx}\ge 1$ on a shape-regular, conforming simplicial mesh with mesh size~$h_{\bx}$. For the temporal discretization, we introduce the space~$S_{h_t}^{p_t}(0,T)$ of piecewise continuous polynomials of degree~$p_t \ge 1$ associated with a mesh of~$(0,T)$ with mesh size~$h_t$, defined by the nodes~$0=:t_0<t_1<\ldots<t_{N_t}:=T$ for $N_t \in \N$. We denote the subspaces of~$S_{h_t}^{p_t}(0,T)$ incorporating zero initial and final conditions, respectively, by~$S_{h_t,0,\bullet}^{p_t}(0,T)$ and~$S_{h_t,\bullet,0}^{p_t}(0,T)$. Furthermore, we define the tensor product spaces
\begin{align*}
    Q_{\bh}^{p_t}(Q_T) & := S_{h_{\bx}}^{p_{\bx}}(\Omega) \otimes S_{h_t}^{p_t}(0,T),
    \\ \partial_t Q_{\bh}^{p_t}(Q_T) & := S_{h_{\bx}}^{p_{\bx}}(\Omega) \otimes \partial_t S_{h_t}^{p_t}(0,T),~
    \\ Q_{\bh,0,\bullet}^{p_t}(Q_T) & := S_{h_{\bx}}^{p_{\bx}}(\Omega) \otimes S_{h_t,0,\bullet}^{p_t}(0,T),
    \\ Q_{\bh,\bullet,0}^{p_t}(Q_T)~ & := S_{h_{\bx}}^{p_{\bx}}(\Omega) \otimes S_{h_t,\bullet,0}^{p_t}(0,T).
\end{align*}
We emphasize that, in our notation for tensor-product discrete spaces and functions, the polynomial degree is specified only for the time discretization. This convention reflects the fact that, while the same discrete spaces are used for both test and trial functions in space, different polynomial degrees are used for them in time.

\noindent
Let $\Pi^{(p_t-1,-1)}_{h_t} :L^2(0,T)\to \partial_t S_{h_t}^{p_t}(0,T)$ be the $L^2$ orthogonal projection. We denote by the same symbol~$\Pi^{(p_t-1,-1)}_{h_t}$ both this operator and its space--time extension, mapping from $L^2(Q_T)$ into $L^2(\Omega) \otimes \partial_t S_{h_t}^{p_t}(0,T)$, which is defined, for $U \in L^2(Q_T)$, by
\begin{equation} \label{eq:5}
    (\Pi_{h_t}^{(p_t-1,-1)} U, \partial_t W_{h_t}^{p_t})_{L^2(Q_T)} = (U, \partial_t W^{p_t}_{h_t})_{L^2(Q_T)} \qquad \text{for all~} W_{h_t}^{p_t} \in L^2(\Omega) \otimes  S_{h_t}^{p_t}(0,T).
\end{equation}
Finally, we let $U_{0,h_{\bx}}, \, V_{0,h_{\bx}}\in S^{p_{\bx}}_{h_{\bx}}(\Omega)$ be approximations of the initial data $U_0$ and $V_0$, which we define as the elliptic and the $L^2$ projection onto~$S_{h_{\bx}}^{p_{\bx}}(\Omega)$, respectively:
\begin{align*}
   (c^2\nabla_{\bx} U_{0,h_{\bx}}, \nabla_{\bx}W_{h_{\bx}})_{L^2(\Omega)} & =(c^2\nabla_{\bx} U_0, \nabla_{\bx}W_{h_{\bx}})_{L^2(\Omega)} \qquad \text{for all~} W_{h_{\bx}}\in S^{p_{\bx}}_{h_{\bx}}(\Omega),
   \\  (V_{0,h_{\bx}},W_{h_{\bx}})_{L^2(\Omega)} & =(V_0,W_{h_{\bx}})_{L^2(\Omega)}\qquad\qquad\quad\ \ \text{for all~}  W_{h_{\bx}}\in S^{p_{\bx}}_{h_{\bx}}(\Omega).
\end{align*}

\subsection{Discrete formulation} \label{sec:23}

We consider the following discretization of problem \eqref{eq:1},
which is a slight modification of the method introduced in~\cite{Zank2021} (see Remark~\ref{rem:21} below):
\begin{tcolorbox}[
    colframe=black!50!white,
    colback=blue!5!white,
    boxrule=0.5mm,
    sharp corners,
    boxsep=0.5mm,
    top=0.5mm,
    bottom=0.5mm,
    right=0.25mm,
    left=0.1mm
]
    \begingroup
    \setlength{\abovedisplayskip}{0pt}
    \setlength{\belowdisplayskip}{0pt}
\, find $U_{\bh}^{p_t} \in Q^{p_t}_{\bh}(Q_T)$ such that $U_{\bh}^{p_t}(\cdot,0) = U_{0,h_{\bx}}$ in~$\Omega$ and\medskip
\begin{equation} \label{eq:6}
\begin{aligned}
	-(\partial_t U_{\bh}^{p_t},  \partial_t W_{\bh}^{p_t})_{L^2(Q_T)} + (c^2 \nabla_{\bx} U_{\bh}^{p_t} , &\nabla_{\bx} \Pi_{h_t}^{(p_t-1,-1)} W_{\bh}^{p_t})_{L^2(Q_T)} 
    \\ & 
     = (F,\Pi_{h_t}^{(p_t-1,-1)} W_{\bh}^{p_t})_{L^2(Q_T)} + (V_{0,h_{\bx}},W_{\bh}^{p_t}(\cdot,0))_{L^2(\Omega)}
\end{aligned}
\end{equation}

\medskip\noindent
\, for all $W_{\bh}^{p_t} \in Q^{p_t}_{\bh,\bullet,0}(Q_T)$.
    \endgroup
\end{tcolorbox}
\noindent
In Section~\ref{sec:24} below, we prove that this method can be equivalently reformulated as the discrete first-order-in-time formulation analyzed in \cite{BalesLasiecka1994, FrenchPeterson1996, Gomez2025} (see~\eqref{eq:10} below). Leveraging this equivalence, stability and convergence properties of formulation~\eqref{eq:6} follow directly from those of the first-order-in-time scheme. 

\begin{remark}[Unstabilized scheme and stabilized scheme of~\cite{Zank2021}] \label{rem:21}
By testing \eqref{eq:1} with a discrete function $W_{\bh}^{p_t} \in Q_{\bh,\bullet,0}^{p_t}(Q_T)$, integrating by parts both in space and time, and imposing the initial condition for $\partial_t U_{\bh}^{p_t}$ in a weak sense, one obtains the following discrete formulation:  find $U_{\bh}^{p_t} \in Q^{p_t}_{\bh}(Q_T)$ such that $U_{\bh}^{p_t}(\cdot,0) = U_{0,h_{\bx}}$ in $\Omega$ and
\begin{equation} \label{eq:7}
\begin{aligned}
	-(\partial_t U_{\bh}^{p_t}, \partial_t W_{\bh}^{p_t})_{L^2(Q_T)} &  + (c^2 \nabla_{\bx} U_{\bh}^{p_t}, \nabla_{\bx}  W_{\bh}^{p_t})_{L^2(Q_T)}
    %\\ &
     = (F,W_{\bh}^{p_t})_{L^2(Q_T)} + (V_0,W_{\bh}^{p_t}(\cdot,0))_{L^2(\Omega)}
\end{aligned}    
\end{equation}
for all $W_{\bh}^{p_t} \in Q^{p_t}_{\bh,\bullet,0}(Q_T)$.
It was established in~\cite{SteinbachZank2020} that method~\eqref{eq:7} is stable if and only if a CFL condition of type $h_t \le C h_{\bx}$ is satisfied, where the constant $C>0$ depends on the space $S^{p_{\bx}}_{h_{\bx}}(\Omega)$, on $p_t$, and on $T$ (see~\cite{FerrariFraschini2024} for the case of discretization in time with maximal regularity splines). To overcome this restriction, it was proposed in~\cite{Zank2021} to insert the projection operator~$\Pi^{(p_t-1,-1)}_{h_t}$ in the grad-grad term. The resulting method was numerically shown to be effective and unconditionally stable.

\smallskip
\noindent
The only difference between the method in~\eqref{eq:6} and that in~\cite{Zank2021} is in the presence of the operator~$\Pi^{(p_t-1,-1)}_{h_t}$ on the right-hand side of the former, which is not included in the latter. As such, the stability results established for~\eqref{eq:6} immediately extend to the method in~\cite{Zank2021}, providing a theoretical justification for the unconditional stability observed there.
\eremk
\end{remark}

\noindent
The scheme in~\eqref{eq:6} has the advantage, with respect to the first-order-in-time scheme in \cite{BalesLasiecka1994, FrenchPeterson1996}, that it involves only the wave field as an unknown, thereby halving the number of unknowns. Moreover, it allows the use of the efficient solver recently developed in~\cite{Zank2025}.

\smallskip
\noindent
While the computation of the projection~$\Pi_{h_t}^{(p_t-1,-1)}$ in the left-hand side
in~\eqref{eq:6} would typically require introducing an auxiliary variable, we show how this can be avoided, leading to a more efficient implementation. This follows from interpreting formulation~\eqref{eq:6} as derived from~\eqref{eq:7} by applying a subquadrature in time to the grad-grad term, as made precise by the following result. 

\begin{proposition}
For all $u_{h_t}^{p_t}, v_{h_t}^{p_t} \in S_{h_t}^{p_t}(0,T)$, it holds true that
\begin{equation} \label{eq:8}
    (u_{h_t}^{p_t}, \Pi_{h_t}^{(p_t-1,-1)} v_{h_t}^{p_t})_{L^2(t_{i-1},t_i)} = \sum_{j=1}^{p_t} u_{h_t}^{p_t}(x_j^{(i)}) v_{h_t}^{p_t}(x_j^{(i)}) \omega_j^{(i)}, \quad \quad i =1,\ldots,N_t,
\end{equation}
where $\{x_j^{(i)}\}_{j=1}^{p_t}$ and $\{\omega_j^{(i)}\}_{j=1}^{p_t}$ are nodes and weights of Gauss--Legendre quadrature with $p_t$ points in $[t_{i-1},t_i]$. 
\end{proposition}
\begin{proof}
As the Gauss--Legendre quadrature formula with~$p_t$ points is exact for polynomials up to degree $2p_t-1$, we have
\begin{equation*}
    (u_{h_t}^{p_t}, \Pi_{h_t}^{(p_t-1,-1)} v_{h_t}^{p_t})_{L^2(t_{i-1},t_i)} = \sum_{j=1}^{p_t} u_{h_t}^{p_t}(x_j^{(i)}) \Pi_{h_t}^{(p_t-1,-1)} v_{h_t}^{p_t}(x_j^{(i)}) \omega_j^{(i)}.
\end{equation*}
To obtain the result, we only need to show that, for all $j=1,\ldots,p_t$ and $i=1,\ldots,N_t$, 
\begin{equation} \label{eq:9}
    v_{h_t}^{p_t}(x_j^{(i)}) = \Pi_{h_t}^{(p_t-1,-1)} v_{h_t}^{p_t}(x_j^{(i)}).
\end{equation}
Let $\{L_j^{(i)}(t)\}_{j \ge 0}$ be the standard Legendre polynomials in $[t_{i-1},t_i]$. Then, from their orthogonality, it follows that
\begin{equation*}
    v_{h_t}^{p_t}(t) - \Pi_{h_t}^{(p_t-1,-1)} v_{h_t}^{p_t}(t) = L_{p_t}^{(i)}(t) \frac{(v_{h_t}^{p_t}, L_{p_t}^{(i)})_{L^2(t_{i-1},t_i)}}{(L_{p_t}^{(i)}, L_{p_t}^{(i)})_{L^2(t_{i-1},t_i)}}, \quad t \in [t_{i-1},t_i].
\end{equation*}
Recalling that $\{x_j^{(i)}\}_{j=1}^{p_t}$ are exactly the zeros of $L_{p_t}^{(i)}(t)$, we deduce \eqref{eq:9}, and thus \eqref{eq:8}.
\end{proof}

\subsection{Equivalence result} \label{sec:24}

The space--time DG--CG discretization of the linear wave equation introduced in \cite{BalesLasiecka1994, FrenchPeterson1996}, which is based on the first-order-in-time formulation~\eqref{eq:4}, reads as follows: 

\begin{tcolorbox}[
    colframe=black!50!white,
    colback=blue!5!white,
    boxrule=0.5mm,
    sharp corners,
    boxsep=0.5mm,
    top=0.5mm,
    bottom=0.5mm,
    right=1mm,
    left=1mm
]
    \begingroup
    \setlength{\abovedisplayskip}{0pt}
    \setlength{\belowdisplayskip}{0pt}
\, find $\widetilde{U}_{\bh}^{p_t}, \widetilde{V}_{\bh}^{p_t} \in Q_{\bh}^{p_t}(Q_T)$ such that $\widetilde{U}_{\bh}^{p_t}(\cdot,0) = U_{0,h_{\bx}}$, $ \widetilde{V}_{\bh}^{p_t}(\cdot,0) = V_{0,h_{\bx}}$ in~$\Omega$, and \medskip
\begin{equation} \label{eq:10}
	\begin{cases}
		(\partial_t \widetilde{U}_{\bh}^{p_t}, \chi_{\bh}^{p_t-1})_{L^2(Q_T)} - (\widetilde{V}_{\bh}^{p_t}, \chi_{\bh}^{p_t-1})_{L^2(Q_T)}=0 \vspace{0.05cm} \\
		(\partial_t \widetilde{V}_{\bh}^{p_t}, \lambda_{\bh}^{p_t-1})_{L^2(Q_T)} + (c^2 \nabla_{\bx} \widetilde{U}_{\bh}^{p_t}, \nabla_{\bx} \lambda_{\bh}^{p_t-1})_{L^2(Q_T)} = (F, \lambda_{\bh}^{p_t-1})_{L^2(Q_T)} 
	\end{cases}
\end{equation}

\medskip
\noindent
\, for all $\chi_{\bh}^{p_t-1},\,\lambda_{\bh}^{p_t-1} \in \partial_t Q_{\bh}^{p_t}(Q_T)$.
    \endgroup
\end{tcolorbox}
\medskip
\noindent
In Theorem~\ref{thm:25} below, we establish an equivalence result between the discrete problems \eqref{eq:6} and \eqref{eq:10}. Before doing so, we rewrite problem \eqref{eq:6} in system form (Lemma~\ref{lem:23}), and define reconstructions in~$Q_{\bh}^{p_t}(Q_T)$ of functions in~$\partial_t Q_{\bh}^{p_t}(Q_T)$ with given initial data (Lemma~\ref{lem:24}).
\begin{lemma} \label{lem:23}
Consider the problem: find $U_{\bh}^{p_t} \in Q_{\bh}^{p_t}(Q_T)$ and $V_{\bh}^{p_t-1} \in \partial_t Q_{\bh}^{p_t}(Q_T)$ such that
\begin{equation} \label{eq:11}
	\begin{cases}
    (\partial_t U^{p_t}_{\bh}, \chi^{p_t-1}_{\bh})_{L^2(Q_T)} - (V_{\bh}^{p_t-1}, \chi_{\bh}^{p_t-1})_{L^2(Q_T)} = 0 & \text{for all~} \chi_{\bh}^{p_t-1} \in  \partial_t Q_{\bh}^{p_t}(Q_T), \vspace{0.05cm}
    \\ -(V_{\bh}^{p_t-1},  \partial_t W_{\bh}^{p_t})_{L^2(Q_T)} + (c^2 \nabla_{\bx} U_{\bh}^{p_t} , \nabla_{\bx} \Pi_{h_t}^{(p_t-1,-1)} W_{\bh}^{p_t})_{L^2(Q_T)} & \vspace{0.05cm}
    \\ \hspace{1cm} = (F,\Pi_{h_t}^{(p_t-1,-1)} W_{\bh}^{p_t})_{L^2(Q_T)}  + (V_{0,h_{\bx}},W_{\bh}^{p_t}(\cdot,0))_{L^2(\Omega)} & \text{for all~} W_{\bh}^{p_t} \in Q^{p_t}_{\bh,\bullet,0}(Q_T)
	\end{cases}
\end{equation}
with the initial condition $U_{\bh}^{p_t}(\cdot,0) = U_{0,h_{\bx}}$ in~$\Omega$. Then, \eqref{eq:6} is equivalent to \eqref{eq:11}, in the sense that $U_{\bh}^{p_t}$ solves~\eqref{eq:6} if and only if $U_{\bh}^{p_t}$ and $V_{\bh}^{p_t-1}= \partial_t U_{\bh}^{p_t}$ solve \eqref{eq:11}.
\end{lemma}
\begin{proof}
As the first equation in~\eqref{eq:11} is simply the identity~$\partial_t U^{p_t}_{\bh}=V_{\bh}^{p_t-1}$, the proof is straightforward.
\end{proof}
\begin{lemma} \label{lem:24}
Given~$V_{\bh}^{p_t-1}\in \partial_t Q_{\bh}^{p_t}(Q_T)$ and~$V_{0,h_{\bx}}\in S_{h_{\bx}}^{p_{\bx}}(\Omega)$, there exists a unique  $\widetilde{V}_{\bh}^{p_t} \in Q_{\bh}^{p_t}(Q_T)$ such that $\widetilde{V}_{\bh}^{p_t}(\cdot,0) = V_{0,h_{\bx}}$ and
\begin{equation*}
    (\widetilde{V}_{\bh}^{p_t}, \partial_t W_{\bh}^{p_t})_{L^2(Q_T)} = (V_{\bh}^{p_t-1}, \partial_t W_{\bh}^{p_t})_{L^2(Q_T)} \quad \text{for all~} W_{\bh}^{p_t} \in Q_{\bh}^{p_t}(Q_T).
\end{equation*}
\end{lemma}
\begin{proof}
    Uniqueness, which also implies existence, is shown by verifying that, if $\widetilde{V}_{\bh}^{p_t}(\cdot,0) = 0$ and
\begin{equation} \label{eq:16}
    (\widetilde{V}_{\bh}^{p_t}, \partial_t W_{\bh}^{p_t})_{L^2(Q_T)} = 0 \quad \text{for all~} W_{\bh}^{p_t} \in Q_{\bh}^{p_t}(Q_T),
\end{equation}
then~$\widetilde{V}_{\bh}^{p_t} = 0$. In order to do so, we observe that the system matrix associated with \eqref{eq:16} has a tensor product structure, composed of a positive definite matrix in space, arising from the $L^2(\Omega)$ inner product, and a matrix in time, associated with the problem
\begin{equation} \label{eq:17}
    \text{find~$\widetilde{v}_{h_t}^{p_t}\in S_{h_t,0,\bullet}^{p_t}(0,t)$ such that}\qquad (\widetilde{v}_{h_t}^{p_t}, \partial_t w_{h_t}^{p_t})_{L^2(0,T)} = 0 \quad \text{for all~} w_{h_t}^{p_t} \in S_{h_t}^{p_t}(0,t),
\end{equation}
that can be also shown to be invertible. This invertibility can be established as in~\cite[Appendix B.1]{FerrariPerugia2025}. In detail, we expand~$\widetilde{v}_{h_t}^{p_t}$ in each interval~$[t_{i-1},t_i]$ of the time mesh as
\begin{equation*}
    \widetilde{v}_{h_t}^{p_t}(t) = \sum_{r=0}^{p_t} c_r^{(i)} L_r^{(i)}(t) \quad \text{for~} t \in [t_{i-1},t_i],
\end{equation*}
where $\{L_r^{(i)}\}_r$ are the Legendre polynomials on~$[t_{i-1},t_i]$, which are orthogonal with respect to the~$L^2(t_{i-1},t_i)$ scalar product, normalized such that $L_r^{(i)}(t_i)=1$. We then have 
\begin{equation*}
    c_r^{(i)} = \frac{(\widetilde{v}_{h_t}^{p_t},  L_r^{(i)})_{L^2(t_{i-1},t_i)}}{(L_r^{(i)},  L_r^{(i)})_{L^2(t_{i-1},t_i)}} \qquad \text{for~$r=0,\ldots, p_t$}.
\end{equation*}
We proceed sequentially, for~$i=1,\ldots,N_t$, testing~\eqref{eq:17} with all functions whose time derivatives coincide with the Legendre polynomials of degree up to~$p_t-1$ over~$[t_{i-1},t_i]$ and are zero outside~$[t_{i-1},t_i]$, and obtain that~$c_r^{(i)} = 0$ for $r=0,\ldots,p_t-1$. From the condition~$\widetilde{v}_{h_t}^{p_t}(t_{i-1})=0$, we also conclude that~$c_{p_t}^{(i)}=0$, and thus~$\widetilde{v}_{h_t}^{p_t}=0$ in~$[t_{i-1},t_i]$, since the zeros of an orthogonal polynomial lie strictly within the interval (see, e.g., \cite[Lemma 3.2]{Iserles2009}).
\end{proof}
\begin{theorem} \label{thm:25}
Let $\widetilde{U}_{\bh}^{p_t}, \widetilde{V}_{\bh}^{p_t} \in Q_{\bh}^{p_t}(Q_T)$ be a solution to the scheme in~\eqref{eq:10}. Then, $U_{\bh}^{p_t} = \widetilde{U}_{\bh}^{p_t} \in Q_{\bh}^{p_t}(Q_T)$
is a solution to \eqref{eq:6}. Conversely, let~$U_{\bh}^{p_t}\in Q_{\bh}^{p_t}(Q_T)$ be a solution of the scheme in~\eqref{eq:6}. Then,
\begin{equation*} 
    \widetilde{U}_{\bh}^{p_t}=U_{\bh}^{p_t}\in Q_{\bh}^{p_t}(Q_T), \quad \widetilde{V}_{\bh}^{p_t} \text{~such that} \quad \Pi_{h_t}^{(p_t-1,-1)} \widetilde{V}_{\bh}^{p_t} = \partial_t U_{\bh}^{p_t}, \quad  \widetilde{V}_{\bh}^{p_t}(\cdot,0) =V_{0,h_{\bx}}
\end{equation*}
is a solution to \eqref{eq:10}.
\end{theorem}
\begin{proof}
Exploiting the equivalence in Lemma~\ref{lem:23}, we study directly the relation between~\eqref{eq:10} and~\eqref{eq:11}. We start by proving that, if $\widetilde{U}_{\bh}^{p_t}, \widetilde{V}_{\bh}^{p_t} \in Q_{\bh}^{p_t}(Q_T)$ solves~\eqref{eq:10}, then 
\begin{equation} \label{eq:12}
    U_{\bh}^{p_t} = \widetilde{U}_{\bh}^{p_t} \in Q_{\bh}^{p_t}(Q_T), \quad V_{\bh}^{p_t-1} = \Pi_{h_t}^{(p_t-1,-1)} \widetilde{V}_{\bh}^{p_t} \in \partial_t Q_{\bh}^{p_t}(Q_T)
\end{equation}
solves~\eqref{eq:11}. We begin by verifying the validity of the first equation in \eqref{eq:11}. For all~$\chi_{\bh}^{p_t-1} \in \partial_t Q_{\bh}^{p_t}(Q_T)$, we have
\begin{equation*}
\begin{aligned}
	(\partial_t U_{\bh}^{p_t}, \chi_{\bh}^{p_t-1})_{L^2(Q_T)} \stackrel{\eqref{eq:12}}{=} (\partial_t \widetilde{U}_{\bh}^{p_t},  \chi_{\bh}^{p_t-1})_{L^2(Q_T)} & \stackrel{\eqref{eq:10}}{=} (\widetilde{V}_{\bh}^{p_t}, \chi_{\bh}^{p_t-1})_{L^2(Q_T)} 
    \\ & \stackrel{\eqref{eq:5}}{=} (\Pi_{h_t}^{(p_t-1,-1)} \widetilde{V}^{p_t}_{\bh}, \chi_{\bh}^{p_t-1})_{L^2(Q_T)} 
     \\ & \stackrel{\eqref{eq:12}}{=} (V_{\bh}^{p_t-1}, \chi_{\bh}^{p_t-1})_{L^2(Q_T)}.
\end{aligned}
\end{equation*}
For the second equation of \eqref{eq:11}, for all $W_{\bh}^{p_t} \in Q_{\bh,\bullet,0}^{p_t}(Q_T)$, we first compute
\begin{equation} \label{eq:13}
\begin{aligned}
     -(V_{\bh}^{p_t-1},  \partial_t W^{p_t}_{\bh} )_{L^2(Q_T)} 
     \stackrel{\eqref{eq:12}}{=} -(\Pi_{h_t}^{(p_t-1,-1)} \widetilde{V}_{\bh}^{p_t}, \partial_t W_{\bh}^{p_t})_{L^2(Q_T)} 
     \stackrel{\eqref{eq:5}}{=} -( \widetilde{V}_{\bh}^{p_t}, \partial_t W_{\bh}^{p_t})_{L^2(Q_T)}.
\end{aligned}
\end{equation}
Then, integrating by parts in time the right-hand side of~\eqref{eq:13}, and using the initial condition~$ \widetilde{V}_{\bh}^{p_t}(\cdot,0) = V_{0,h_{\bx}}$, as well as the zero final condition on the test functions~$W_{\bh}^{p_t}$, we get
 \begin{equation} \label{eq:14}
\begin{aligned}
-(V_{\bh}^{p_t-1},  \partial_t W^{p_t}_{\bh})_{L^2(Q_T)} 
     & = (\partial_t \widetilde{V}_{\bh}^{p_t}, W_{\bh}^{p_t})_{L^2(Q_T)} + (V_{0,h_{\bx}},W_{\bh}^{p_t}(\cdot,0))_{L^2(\Omega)}
    \\ & \stackrel{\eqref{eq:5}}{=} (\partial_t \widetilde{V}_{\bh}^{p_t}, \Pi_{h_t}^{(p_t-1,-1)} W_{\bh}^{p_t})_{L^2(Q_T)} + (V_{0,h_{\bx}},W_{\bh}^{p_t}(\cdot,0))_{L^2(\Omega)}
    \\ & \stackrel{\eqref{eq:10}}{=} -(c^2 \nabla_{\bx} \widetilde{U}_{\bh}^{p_t}, \nabla_{\bx} \Pi_{h_t}^{(p_t-1,-1)} W_{\bh}^{p_t})_{L^2(Q_T)} 
    \\ & \qquad + (F,\Pi_{h_t}^{(p_t-1,-1)} W_{\bh}^{p_t})_{L^2(Q_T)}+ (V_{0,h_{\bx}},W_{\bh}^{p_t}(\cdot,0))_{L^2(\Omega)}
    \\ & \stackrel{\eqref{eq:12}}{=} -(c^2 \nabla_{\bx} U_{\bh}^{p_t}, \nabla_{\bx} \Pi_{h_t}^{(p_t-1,-1)} W_{\bh}^{p_t})_{L^2(Q_T)} 
    \\ & \qquad + (F,\Pi_{h_t}^{(p_t-1,-1)} W_{\bh}^{p_t})_{L^2(Q_T)}+ (V_{0,h_{\bx}},W_{\bh}^{p_t}(\cdot,0))_{L^2(\Omega)},
\end{aligned}
\end{equation}
and the proof of the first claim is complete. 

\noindent
Conversely, if~$U_{\bh}^{p_t} \in Q_{\bh}^{p_t}(Q_T)$, $V_{\bh}^{p_t-1} \in \partial_t Q_{\bh}^{p_t}(Q_T)$ solves~\eqref{eq:11}, then
\begin{equation} \label{eq:15}
    \widetilde{U}_{\bh}^{p_t}=U_{\bh}^{p_t}\in Q_{\bh}^{p_t}(Q_T), \quad \widetilde{V}_{\bh}^{p_t} \text{~such that} \quad \Pi_{h_t}^{(p_t-1,-1)} \widetilde{V}_{\bh}^{p_t} = V_{\bh}^{p_t-1}, \quad  \widetilde{V}_{\bh}^{p_t}(\cdot,0) = V_{0,h_{\bx}}
\end{equation}
solves~\eqref{eq:10}. Existence and uniqueness of~$\widetilde{V}_{\bh}^{p_t}$ as in~\eqref{eq:15} follow from Lemma \ref{lem:24}. With the existence of~$\widetilde{V}_{\bh}^{p_t}$, the first equation in \eqref{eq:10} is readily verified: for all $\chi_{\bh}^{p_t-1} \in \partial_t Q_{\bh}^{p_t}(Q_T)$,
\begin{equation*}
\begin{aligned}
	(\partial_t \widetilde{U}_{\bh}^{p_t}, \chi_{\bh}^{p_t-1})_{L^2(Q_T)} \stackrel{\eqref{eq:15}}{=} (\partial_t U_{\bh}^{p_t}, \chi_{\bh}^{p_t-1})_{L^2(Q_T)} & \stackrel{\eqref{eq:11}}{=} (V_{\bh}^{p_t-1}, \chi_{\bh}^{p_t-1})_{L^2(Q_T)}
    \\ & \stackrel{\eqref{eq:15}}{=} (\Pi_{h_t}^{(p_t-1,-1)} \widetilde{V}^{p_t}_{\bh}, \chi_{\bh}^{p_t-1})_{L^2(Q_T)} 
     \\ & \stackrel{\eqref{eq:5}}{=} (\widetilde{V}_{\bh}^{p_t}, \chi_{\bh}^{p_t-1})_{L^2(Q_T)}.
\end{aligned}
\end{equation*}
For the second equation of \eqref{eq:10}, let $\lambda_{\bh}^{p_t-1} \in \partial_t Q_{\bh}^{p_t}(Q_T)$. Similarly as above, there exists a (unique) $W_{\bh}^{p_t} \in Q_{\bh,\bullet,0}^{p_t}(Q_T)$ such that $\Pi_{h_t}^{(p_t-1,-1)} W_{\bh}^{p_t} = \lambda_{\bh}^{p_t-1}$. Then, integrating by parts in time and using initial and final conditions as in~\eqref{eq:14}, we compute
\begin{equation*}
\begin{aligned}
    (\partial_t \widetilde{V}_{\bh}^{p_t}, \lambda^{p_t-1}_{\bh})_{L^2(Q_T)}
     & = (\partial_t \widetilde{V}_{\bh}^{p_t}, \Pi_{h_t}^{(p_t-1,-1)} W^{p_t}_{\bh})_{L^2(Q_T)}
    \\ & \stackrel{\eqref{eq:5}}{=} (\partial_t \widetilde{V}_{\bh}^{p_t},  W^{p_t}_{\bh})_{L^2(Q_T)}
    \\ & = -(\widetilde{V}_{\bh}^{p_t}, \partial_t W_{\bh}^{p_t})_{L^2(Q_T)} - (V_{0,h_{\bx}},W_{\bh}^{p_t}(\cdot,0))_{L^2(\Omega)}
    \\ & \stackrel{\eqref{eq:5}}{=} -(\Pi_{h_t}^{(p_t-1,-1)} \widetilde{V}_{\bh}^{p_t}, \partial_t W_{\bh}^{p_t})_{L^2(Q_T)} - (V_{0,h_{\bx}},W_{\bh}^{p_t}(\cdot,0))_{L^2(\Omega)}
    \\ & \stackrel{\eqref{eq:16}}{=} -(V_{\bh}^{p_t-1}, \partial_t W_{\bh}^{p_t})_{L^2(Q_T)} - (V_{0,h_{\bx}},W_{\bh}^{p_t}(\cdot,0))_{L^2(\Omega)}
    \\ & \stackrel{\eqref{eq:11}}{=} -(c^2 \nabla_{\bx} U_{\bh}^{p_t} , \nabla_{\bx} \Pi_{h_t}^{(p_t-1,-1)} W_{\bh}^{p_t})_{L^2(Q_T)} 
    \\ & \qquad \hspace{0.1cm}  + (F,\Pi_{h_t}^{(p_t-1,-1)} W_{\bh}^{p_t})_{L^2(Q_T)}
    \\ & = -(c^2 \nabla_{\bx} U_{\bh}^{p_t}, \nabla_{\bx} \lambda_{\bh}^{p_t-1})_{L^2(Q_T)}  + (F,\lambda_{\bh}^{p_t-1})_{L^2(Q_T)}
    \\ & \stackrel{\eqref{eq:15}}{=} -(c^2 \nabla_{\bx} \widetilde{U}_{\bh}^{p_t}, \nabla_{\bx} \lambda_{\bh}^{p_t-1})_{L^2(Q_T)}  + (F,\lambda_{\bh}^{p_t-1})_{L^2(Q_T)},    
\end{aligned}
\end{equation*}
which completes the proof.
\end{proof}
\begin{remark}
The equivalence established in Theorem \ref{thm:25} is similar 
to the result in~\cite{Cockbrun2025}, which shows that the discretizations of the derivative by the continuous Galerkin method (see \cite{Hulme72a}) and by the discontinuous Galerkin method (see, e.g., \cite{SchotzauSchwab2000}) are the same up to a reconstruction. In fact, the second equation in \eqref{eq:11} can also be recast as a particular DG-in-time method (see the second equation in \eqref{eq:35} below). However, the flux definition used here differs from that adopted in \cite{Cockbrun2025}.
\end{remark}
\begin{remark}
We emphasize that the second-order-in-time DG--CG formulation proposed and analyzed in~\cite{Walkington2014, DongMascottoWang2024, DongGeorgoulisMascottoWang2025} is not directly related to the method consider in this paper and to the first-order-in-time schemes of \cite{BalesLasiecka1994,FrenchPeterson1996}. That method is obtained from \eqref{eq:2} without integration by parts in time.
\end{remark}

\subsection{Theoretical properties} \label{sec:25}

Owing to the equivalence result implied by Theorem \ref{thm:25}, the main theoretical results established for the scheme in~\eqref{eq:10} immediately carry over to~\eqref{eq:6}. We first summarize these results for~\eqref{eq:10}, for which we refer to~\cite{FrenchPeterson1996} and~\cite{Gomez2025}, and then state the corresponding results for~\eqref{eq:6} in Theorem~\ref{th:28}. 

\medskip\noindent
{\bf \emph{i)} Well-posedness} (\cite[Theorem~1]{FrenchPeterson1996}, \cite[Theorem~3.5]{Gomez2025}){\bf .} There exists a unique  solution $(\widetilde{U}_{\bh}^{p_t},\widetilde{V}_{\bh}^{p_t})$ of~\eqref{eq:10} and it satisfies 
\begin{equation*}
    \| \widetilde{V}_{\bh}^{p_t}\|_{C([0,T];L^2(\Omega))} + \| c \nabla_{\bx} \widetilde{U}_{\bh}^{p_t}\|_{C([0,T];L^2(\Omega)^d)} \lesssim \| V_0 \|_{L^2(\Omega)} + \| c \nabla_{\bx} U_0 \|_{L^2(\Omega)^d} + \| F \|_{L^1(0,T;L^2(\Omega)},
\end{equation*}
where the hidden constant only depends on $p_t$.

\smallskip
\noindent
For the error estimates, we adopt the setting of~\cite{Gomez2025} and formulate the following additional assumption.
\begin{assumption}\label{ass:27}
Let $\Omega$ satisfy the elliptic regularity assumption, i.e., if $\varphi \in H_0^1(\Omega)$ and $\Delta_{\bx} \varphi \in L^2(\Omega)$, then $\varphi \in H^2(\Omega)$. Moreover, let the data satisfy $c \in C(\overline{\Omega})\cap W^1_\infty(\Omega)$ and, for some $r>\max\{1,d/2\}$, $U_0\in H^r(\Omega)$ and~$V_0\in H^r(\Omega)\cap H^1_0(\Omega)$, and let the solution of the weak formulation~\eqref{eq:2} satisfy $U \in H^2(0,T;H^r(\Omega))$ with $\nabla_{\bx} \cdot (c^2 \nabla_{\bx} U) \in H^1(0,T;L^2(\Omega))$. 
\end{assumption}

\medskip\noindent
{\bf \emph{ii)} Error estimates} (\cite[Theorem~7]{FrenchPeterson1996}, \cite[Theorem~4.9]{Gomez2025}){\bf .} Let  Assumption~\ref{ass:27} hold. Then, provided that $V_0 \in H^{\ell+1}(\Omega)$ and
\begin{equation} \label{eq:18}
    U \in W_1^2(0,T;H^{\ell+1}(\Omega)) \cap W_1^{m+2}(0,T;H^1(\Omega)) \cap W_1^{m+1}(0,T;H^2(\Omega)) \cap C^{m+2}([0,T];L^2(\Omega))
\end{equation}
are satisfied for given~$\ell$ and $m$ with $1 \le \ell \le p_{\bx}$ and $1 \le m \le p_t$, 
the following error estimates hold:
\begin{align}\nonumber
   \|\partial_t U -  \widetilde{V}_{\bh}^{p_t}\|_{C([0,T];L^2(\Omega))} 
   \lesssim & \, h_{\bx}^{\ell+1}\Big( \|\partial_t U\|_{C([0,T];H^{\ell+1}(\Omega))} + \|V_0\|_{H^{\ell+1}(\Omega)}+ \|\partial_{tt} U\|_{L^1(0,T;H^{\ell+1}(\Omega))}\Big)
    \\ \nonumber & + h_t^{m+1}\Big( \|\partial_t^{(m+2)} U\|_{C([0,T];L^2(\Omega))} + \|\nabla_{\bx} \cdot (c^2 \nabla_{\bx} \partial_t^{(m+1)} U)\|_{L^1(0,T;L^2(\Omega))}
    \\ \nonumber &\qquad\qquad + \|c \nabla_{\bx} \partial_t^{(m+2)} U\|_{L^1(0,T;L^2(\Omega)^d)} \Big),
   \\ \nonumber \|c \nabla_{\bx} (U-\widetilde{U}_{\bh})\|_{C([0,T];L^2(\Omega))} 
   \lesssim & \, h_{\bx}^\ell \Big(\|U\|_{C([0,T];H^{\ell+1}(\Omega))}  + h_{\bx} \|V_0\|_{H^{\ell+1}(\Omega)} + h_{\bx}\|\partial_{tt}U\|_{L^1(0,T;H^{\ell+1}(\Omega))} \Big)
    \\ \nonumber & + h_t^{m+1}\Big( \|c\nabla_{\bx} \partial_t^{(m+1)} U\|_{C([0,T];L^2(\Omega)^d)} 
    \\ \label{eq:19} & \qquad\qquad + \|\nabla_{\bx} \cdot (c^2 \nabla_{\bx} \partial_t^{(m+1)} U)\|_{L^1(0,T;L^2(\Omega))}
     \\ \nonumber & \qquad\qquad+ \|c \nabla_{\bx} \partial_t^{(m+2)} U\|_{L^1(0,T;L^2(\Omega)^d)} \Big),
     \\ \nonumber \|U-\widetilde{U}_{\bh}\|_{C([0,T];L^2(\Omega))} 
    \lesssim & \, h_{\bx}^{\ell+1}\Big( \|U\|_{C([0,T];H^{\ell+1}(\Omega))} + \|V_0\|_{H^{\ell+1}(\Omega)} + \|\partial_{tt}U\|_{L^1(0,T;H^{\ell+1}(\Omega))}\Big) 
    \\ \nonumber & + h_{\bx}^{m+1}\Big( \|\partial_t^{(m+2)} U\|_{C([0,T];L^2(\Omega))} + \|\nabla_{\bx} \cdot (c^2 \nabla_{\bx} \partial_t^{(m+1)} U)\|_{L^1(0,T;L^2(\Omega))}
    \\ \nonumber &\qquad\qquad + \|c \nabla_{\bx} \partial_t^{(m+2)} U\|_{L^1(0,T;L^2(\Omega)^d)} \Big),
\end{align}
with hidden constants independent of~$h_{\bx}$ and~$h_{t}$.

\medskip\noindent
{\bf \emph{iii)} Energy conservation.} Let $\widetilde{U}_{\bh}^{p_t}, \widetilde{V}_{\bh}^{p_t}$ be a solution of \eqref{eq:10}. Then, for~$F=0$, the following energy conservation property is satisfied: for all~$j=0,\ldots, N_t$, 
\begin{equation*} 
\begin{aligned}
   \frac{1}{2} \left( 
   \| \widetilde{V}_{\bh}^{p_t}(\cdot,t_j)\|^2_{L^2(\Omega)} + \| c\nabla_{\bx} \widetilde{U}_{\bh}^{p_t}(\cdot,t_j)\|^2_{L^2(\Omega)^d} 
   \right) = \frac{1}{2} \left( \| V_{0,h_{\bx}}\|^2_{L^2(\Omega)} + \| c\nabla_{\bx} U_{0,h_{\bx}}\|^2_{L^2(\Omega)^d}
   \right).
\end{aligned}
\end{equation*}
This identity can be derived by testing \eqref{eq:10}, for $j=0,\ldots,N_t$, with 
\begin{equation*}
    \chi_{\bh}^{p_t-1}(\bx,t) = -\partial_t \widetilde{V}_{\bh}^{p_t}(\bx,t) \mathds{1}_{[0,t_j]}(t) \quad \text{and} \quad \lambda_{\bh}^{p_t-1}(\bx,t) = \partial_t \widetilde{U}_{\bh}^{p_t}(\bx,t) \mathds{1}_{[0,t_j]}(t),
\end{equation*}
and adding the resulting two equations. This gives
\begin{equation*}
    (\widetilde{V}_{\bh}^{p_t}, \partial_t \widetilde{V}_{\bh}^{p_t})_{L^2(\Omega \times [0,t_j])} + (c^2 \nabla_{\bx} \widetilde{U}_{\bh}^{p_t}, \partial_t \widetilde{U}_{\bh}^{p_t})_{L^2(\Omega \times [0,t_j])} =0,
\end{equation*}
from which~\eqref{eq:28} readily follows.

\medskip
\noindent
We deduce now the corresponding properties for~\eqref{eq:6}.
\begin{theorem}\label{th:28}
There exists a unique solution $U_{\bh}^{p_t}$ of~\eqref{eq:6} and it satisfies 
\begin{equation*}
    T^{-1/2}\| \partial_t U_{\bh}^{p_t}\|_{L^2(Q_T)}
    + \| c \nabla_{\bx} U_{\bh}^{p_t}\|_{C([0,T];L^2(\Omega)^d)} \lesssim \| V_0 \|_{L^2(\Omega)} + \| c \nabla_{\bx} U_0 \|_{L^2(\Omega)^d} + \| F \|_{L^1(0,T;L^2(\Omega))},
\end{equation*}
where the hidden constant only depends on $p_t$. Moreover, under Assumption~\ref{ass:27}, if $V_0 \in H^{\ell+1}(\Omega)$ and $U$ has the regularity in \eqref{eq:18} for given~$\ell$ and $m$ with $1 \le \ell \le p_{\bx}$ and $1 \le m \le p_t$, the following error estimates hold:
\begin{equation} \label{eq:20}
\begin{aligned}
   \|c \nabla_{\bx} (U-U_{\bh}^{p_t})\|_{C([0,T];L^2(\Omega))} \lesssim & \, h_{\bx}^\ell  + h_t^{m+1},
   \\ \|U-U_{\bh}^{p_t}\|_{C([0,T];L^2(\Omega))} \lesssim & \, h_{\bx}^{\ell+1}+ h_t^{m+1}.
\end{aligned}
\end{equation}
Here, we reported only the rates, omitting the dependence on~$U$ and~$V_0$, which are exactly as in the last two estimates in~\eqref{eq:19}. Furthermore, we have 
\begin{equation} \label{eq:21}
\begin{aligned}
    \| \partial_t U- \partial_t U_{\bh}^{p_t}\|_{L^2(Q_T))} 
    & \lesssim  h_{\bx}^{\ell+1}\Big( \|\partial_t U\|_{C([0,T];H^{\ell+1}(\Omega))} + \|V_0\|_{H^{\ell+1}(\Omega)}+ \|\partial_{tt} U\|_{L^1(0,T;H^{\ell+1}(\Omega))}\Big)
    \\ & \,\, + h_t^{\min\{m+1,p_t\}}\|\partial_t^{(m+2)} U\|_{C([0,T];L^2(\Omega))} 
    \\ & \,\, + h_t^{m+1}\Big( \|\nabla_{\bx} \cdot (c^2 \nabla_{\bx} \partial_t^{(m+1)} U)\|_{L^1(0,T;L^2(\Omega))} + \|c \nabla_{\bx} \partial_t^{(m+2)} U\|_{L^1(0,T;L^2(\Omega)^d)} \Big)
    \\ & \lesssim h_{\bx}^{\ell+1}  + h_t^{\min\{m+1,p_t\}}.
\end{aligned}
\end{equation}
Finally, for~$F=0$, the following energy conservation property is satisfied: for all~$j=0,\ldots, N_t$, 
\begin{equation} \label{eq:21b}
\begin{aligned}
   \frac{1}{2} \left( 
   \| \widetilde{V}_{\bh}^{p_t}(\cdot,t_j)\|^2_{L^2(\Omega)} + \| c\nabla_{\bx} U_{\bh}^{p_t}(\cdot,t_j)\|^2_{L^2(\Omega)^d} 
   \right) = \frac{1}{2} \left( \| V_{0,h_{\bx}}\|^2_{L^2(\Omega)} + \| c\nabla_{\bx} U_{0,h_{\bx}}\|^2_{L^2(\Omega)^d}
   \right).
\end{aligned}
\end{equation}
where $\widetilde{V}_{\bh}^{p_t}$ is the unique function in~$Q_{\bh}^{p_t}(Q_T)$ such that
\begin{equation} \label{eq:22}
    \Pi_{h_t}^{(p_t-1,-1)} \widetilde{V}_{\bh}^{p_t} = \partial_t U_{\bh}^{p_t} \quad \text{and}\quad  \widetilde{V}_{\bh}^{p_t}(\cdot,0) = V_{0,h_{\bx}},
\end{equation}
\end{theorem}
\begin{proof}
From Theorem \ref{thm:25}, we deduce that  $U_{\bh}^{p_t} \in Q_{\bh}^{p_t}(Q_T)$ is a solution of~\eqref{eq:6} if and only if the pair $(\widetilde{U}_{\bh}^{p_t},\widetilde{V}_{\bh}^{p_t})\in Q_{\bh}^{p_t}(Q_T) \times Q_{\bh}^{p_t}(Q_T)$, with $\widetilde{U}_{\bh}^{p_t} = U_{\bh}^{p_t}$ and $\widetilde{V}_{\bh}^{p_t}$ satisfying \eqref{eq:22}, is a solution of~\eqref{eq:10}. Consequently, the well-posedness, error estimates \eqref{eq:20}, and energy conservation properties immediately follow from those of~\eqref{eq:10}. It remains to prove~\eqref{eq:21}. We compute, using \eqref{eq:22} and the stability of the $L^2$ projection, 
\begin{equation*}
\begin{aligned}
        \| \partial_t U - \partial_t U_{\bh}^{p_t} \|_{L^2(Q_T)} & \le  \| \partial_t U - \Pi_{h_t}^{(p_t-1,-1)} \partial_t U \|_{L^2(Q_T)} + \| \Pi_{h_t}^{(p_t-1,-1)} \partial_t U - \partial_t U_{\bh}^{p_t} \|_{L^2(Q_T)}
        \\ & =  \| \partial_t U - \Pi_{h_t}^{(p_t-1,-1)} \partial_t U \|_{L^2(Q_T)} + \| \Pi_{h_t}^{(p_t-1,-1)}( \partial_t U - \widetilde{V}_{\bh}^{p_t}) \|_{L^2(Q_T)}
        \\ & \le  \| \partial_t U - \Pi_{h_t}^{(p_t-1,-1)} \partial_t U \|_{L^2(Q_T)} + \| \partial_t U -  \widetilde{V}_{\bh}^{p_t} \|_{L^2(Q_T)}
        \\ & \le  \| \partial_t U - \Pi_{h_t}^{(p_t-1,-1)} \partial_t U \|_{L^2(Q_T)} + T^{1/2} \| \partial_t U -  \widetilde{V}_{\bh}^{p_t} \|_{C([0,T];L^2(\Omega))}.
\end{aligned}
\end{equation*}
Then, we conclude with the approximation property of the $L^2$ projection and the first estimate in \eqref{eq:19}.
\end{proof}
\begin{remark}[Reconstruction $\widetilde{V}_{\bh}^{p_t}$]
For the reconstruction~$\widetilde{V}_{\bh}^{p_t}$ of~$\partial_t U_{\bh}^{p_t}$ defined in~\eqref{eq:22}, the stability estimate
\begin{equation*}
    \| \widetilde{V}_{\bh}^{p_t}\|_{C([0,T];L^2(\Omega))} + \| c \nabla_{\bx} U_{\bh}^{p_t}\|_{C([0,T];L^2(\Omega)^d)} \lesssim \| V_0 \|_{L^2(\Omega)} + \| c \nabla_{\bx} U_0 \|_{L^2(\Omega)^d} + \| F \|_{L^1(0,T;L^2(\Omega))}
\end{equation*}
and the error estimate
\begin{equation*}
    \|\partial_t U - \widetilde{V}_{\bh}^{p_t}\|_{C([0,T];L^2(\Omega))} \lesssim h_{\bx}^{\ell+1} + h_t^{m+1}
\end{equation*}
immediately follow from those for formulation \eqref{eq:10}.
\end{remark}
\begin{remark} [Reconstruction $\USG$]
Following \cite[\S 4.4]{Gomez2025}, a postprocessed approximation $\USG \in Q_{\bh}^{p_t+1}(Q_T)$ of $U$ from the numerical solution $U_{\bh}^{p_t}$ of \eqref{eq:6} can be obtained with the formula
\begin{equation*}
    \USG(\bx,t) := U_{0,h_{\bx}}(\bx) + \int_0^t \widetilde{V}_{\bh}^{p_t}(\bx,s) \dd s, \quad (\bx,t) \in Q_T,
\end{equation*}
with $\widetilde{V}^{p_t}_{\bh}$ as in \eqref{eq:22}. For $p_t > 1$, and provided that $U$ is sufficiently smooth, it was shown in \cite[Theorem~4.12]{Gomez2025} that
\begin{equation*}
    \| U - \USG\|_{C([0,T];L^2(\Omega))} \apprle h_{\bx}^{p_{\bx}+1} + h_t^{p_t+2}.
\end{equation*}
\end{remark}

\subsubsection{Numerical results} \label{sec:251}

\noindent
For a detailed campaign of numerical experiments validating the stability results and convergence estimates stated in Section \ref{sec:25} for formulation \eqref{eq:10}, we refer the reader to \cite{Gomez2025}. Since \eqref{eq:6} is equivalent to \eqref{eq:10}, the same stability properties and convergence behavior are confirmed there as well. We also refer to~\cite{Zank2021} and~\cite[\S 7]{Zank2025} for numerical results derived directly with the variant of the second-order formulation~\eqref{eq:6} introduced in~\cite{Zank2021} (see Remark~\ref{rem:21}). In this section, we focus on verifying the energy conservation property given in \eqref{eq:21b}.

\noindent
Consider the linear wave equation \eqref{eq:1} on $Q_T = \Omega \times (0,T)$, for $\Omega = (-30,30)$ and $T=10$, with data
\begin{equation*}
    c(x)=1, \quad F(x,t) = 0, \quad U_0(x) = \omega(x+1)S(x+1), \quad V_0(x) = -\omega'(x+1) S(x+1) - \omega(x+1) S'(x+1),
\end{equation*}
where
\begin{equation*}
    \omega(s) := e^{-20(s-0.1)^2}- e^{-20(s+0.1)^2}, \qquad S(s) := \frac{1}{1+e^{-30s}}.
\end{equation*}
The exact solution is
\begin{equation*}
    U(x,t) = \omega(x-t+1) S(x-t+1).
\end{equation*}
We can assume homogeneous Dirichlet boundary conditions, since
\begin{equation*}
    \max_{t \in (0,10), \,\, x \in \{-30,30\}} |U(x,t)| \le 10^{-16}.
\end{equation*}
We test the formulation in~\eqref{eq:10} on a uniform mesh with $N_t = 128$ time elements and $N_x = 384$ spatial elements, and using polynomial degrees $p_t = p_x = 1$ and $p_t = p_x = 2$ . For each time step $t_j$, $j = 0,\ldots,N_t$, we evaluate the discrete energy
\begin{equation*}
    E_{\bh}^{p_t}(t_j) := \frac{1}{2} \left(\| \widetilde{V}_{\bh}^{p_t}(\cdot,t_j) \|^2_{L^2(-30,30)} + \| \nabla_{\bx} U_{\bh}^{p_t}(\cdot,t_j) \|^2_{L^2(-30,30)} \right).
\end{equation*}
We recall that $\widetilde{V}^{p_t}_{\bh}$ is the reconstruction of~$\partial_t U^{p_t}_{\bh}$ defined in \eqref{eq:22}, and $U_{h_x,0}$ and $ V_{h_x,0}$ are defined using the elliptic and $L^2$ projections, respectively. We consider the quantity $|E_{\bh}^{p_t}(t_j) - E_{\bh}^{p_t}(0)|$ for $j=1,\ldots,N_t$. According to \eqref{eq:21b}, this difference is expected to be of the order of machine precision. The numerical results, reported in Figure \ref{fig:1}, are in good agreement with this expectation. 

\begin{figure}[h!]
    \centering
    \begin{subfigure}[b]{0.485\textwidth}
    \centering
         \includegraphics[width=\textwidth]{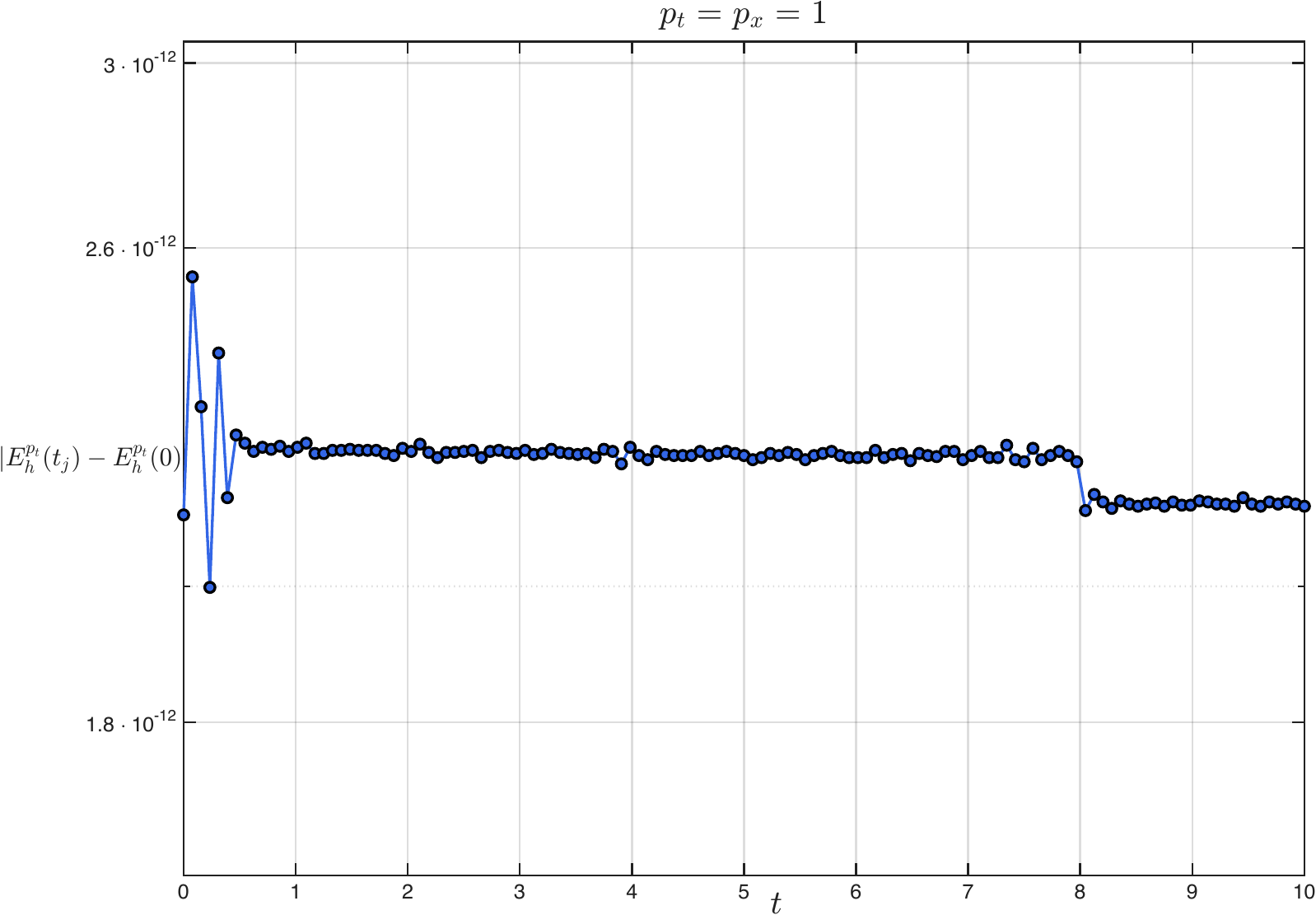}
     \end{subfigure}
    \begin{subfigure}[b]{0.50\textwidth}
    \centering
         \includegraphics[width=\textwidth]{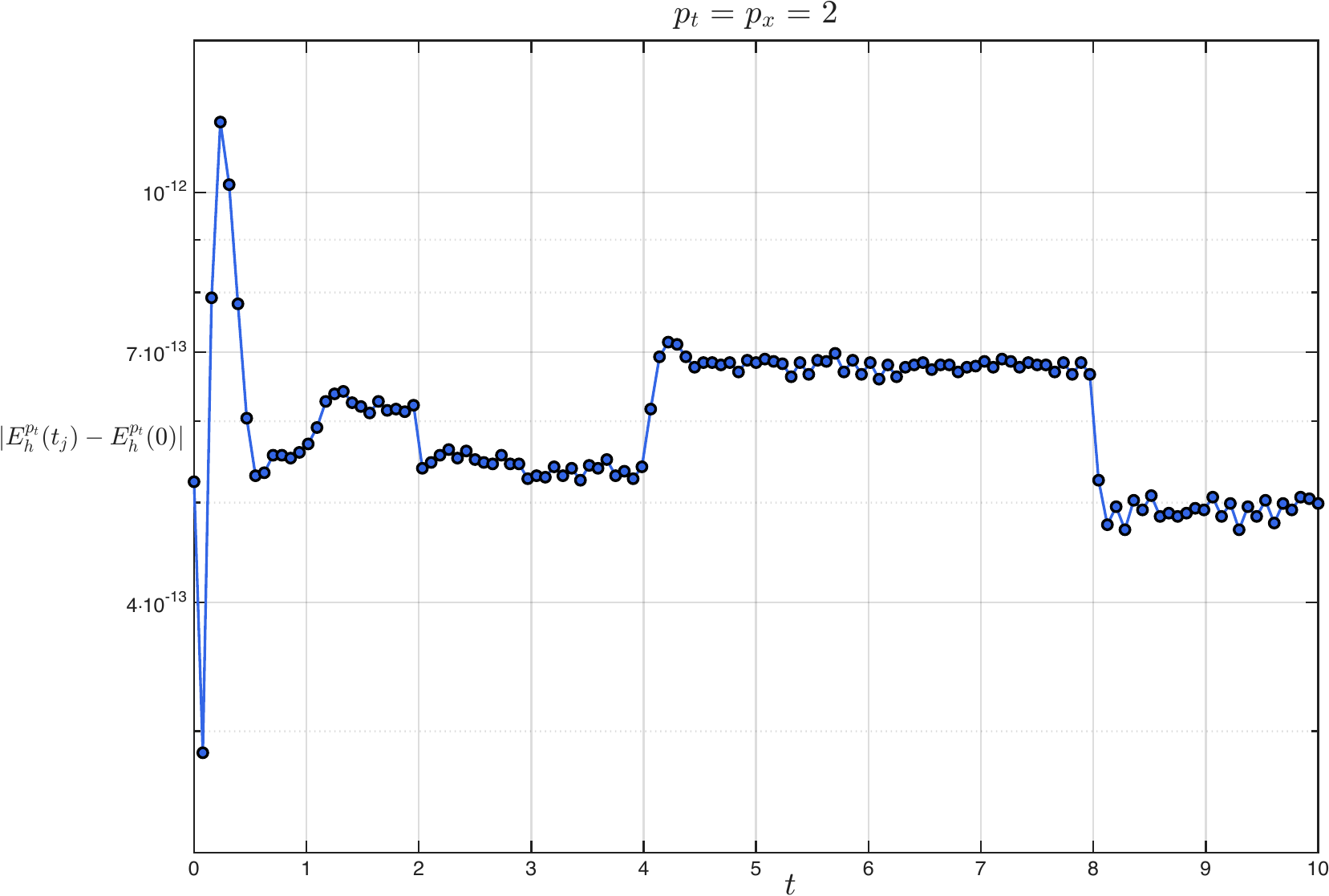}
     \end{subfigure}
\caption{Evolution of the discrete energy error $|E_{\bh}^{p_t}(t_j)-E_{\bh}^{p_t}(0)|$ using polynomial degrees \(p_t=p_x=1\) (left) and \(p_t=p_x=2\) (right).}
\label{fig:1}
\end{figure}

\noindent 
We emphasize that, to observe discrete energy conservation, it is crucial to compute the energy using the reconstruction $\widetilde{V}_{\bh}^{p_t}$ rather than $\partial_t U_{\bh}^{p_t}$. As shown, for example, in~\cite[Remark 4.2.33]{Zank2020}, the conservation of the discrete energy defined using $\partial_t U_{\bh}^{p_t}$ is not achieved to machine precision, as in Figure~\ref{fig:1}, but only up to the same order of accuracy as the numerical method.

\section{Extension to the semilinear case} \label{sec:3}

In this section, we consider \emph{semilinear} wave equations. While the extension of the first-order-in-time DG--CG formulation \eqref{eq:10} of \cite{BalesLasiecka1994,FrenchPeterson1996} to this setting was already studied in \cite{KarakashianMakridakis2005}, here we show that the second-order-in-time method \eqref{eq:6} can be similarly adapted. The equivalence of the resulting second-order- and first-order-in-time formulations carries over by the same arguments as in the linear case. We also introduce two symplectic versions of the second-order-in-time method, one employing the Gauss--Lobatto quadrature rule in time, and the other the Gauss--Legendre one.

\subsection{Semilinear wave equations}
We consider the following initial-boundary-value problem, involving a nonlinear term~$g(U)$ with~$g(0)=0$:
\begin{equation} \label{eq:23}
	\begin{cases}
		\partial_{t}^2 U - \div_{\bx}(c^2 \nabla_{\bx} U) + \NL(U) =  F & \text{in~} Q_T, 
		\\ U = 0 &  \text{on~} \partial \Omega \times (0,T),
		\\ U(\cdot,0) = U_0, \quad \partial_t U(\cdot,t)_{|_{t=0}} = V_0 & \text{in~} \Omega.
	\end{cases}
\end{equation}
The corresponding first-order-in-time problem reads:
\begin{equation} \label{eq:24}
	\begin{cases}
		\partial_t U - V = 0 & \text{in~} Q_T,
        \\ \partial_t V - \div_{\bx}(c^2 \nabla_{\bx} U) + \NL(U) =  F & \text{in~} Q_T,
		\\ U = 0 &  \text{on~} \partial \Omega \times (0,T),
		\\ U(\cdot,0) = U_0, \quad V(\cdot,0) = V_0 & \text{in~} \Omega.
	\end{cases}
\end{equation}
For later convenience, we introduce
\begin{equation}\label{eq:25}
    G(U) := \int_0^U \NL(V) \dd V.
\end{equation}
\begin{remark}[Examples of nonlinearities]
This framework includes several key models:
\begin{itemize}
    \item $\NL(U) = c^2 \sin U$: the \textit{sine-Gordon equation},
    \item $\NL(U) = c^2 U$: the \textit{linear Klein–Gordon equation},
    \item $\NL(U) = c^4 U + c^2 |U|^{\rho} U $,  with $\rho$ in the subcritical regime, i.e., $\rho>0$ if $d=1,2$ and $0 < \rho \le 4$ if $d=3$: the \textit{nonlinear Klein-Gordon equation with defocusing nonlinearity}.
\end{itemize}
\end{remark}
\begin{remark} [Well-posedness results]\label{rem:semilin_wp}
Well-posedness results for~\eqref{eq:23} date back at least to~\cite{Lions1969}. For $\Omega\subset\R^d$ a bounded, Lipschitz domain, in the specific case $\NL(U) = |U|^{\rho}U$, with $\rho > 0$ if $d=2$, or $0 < \rho \leq 2$ if $d=3$, it was established in \cite[Theorems~1.1 and 1.2]{Lions1969} that, for~$c=1$, $F\in L^2(Q_T)$, $U_0\in H^1_0(\Omega)\cap L^{\rho+2}(\Omega)$, and~$V_0\in L^2(\Omega)$, there exists a unique weak solution~$U$ to~\eqref{eq:23} in~$L^\infty(0,T;H^1_0(\Omega)\cap L^{\rho+2}(\Omega))\cap W^{1,\infty}(0,T;L^2(\Omega))$. The result extends to the case of~$\NL$ monotone of polynomial growth, for a polynomial degree sufficiently small (see, e.g.~\cite{LieroStefanelli2013}).

\noindent
For~$\Omega$ a $C^2$-regular domain, assuming $\NL \in W^{1,\infty}(\R)$ with
\begin{align*}
    & z \NL(z) \ge 0 \quad \text{for all~} z\in\R, 
    \\ & G(U_0)\in L^1(\Omega), \quad \text{with~$G$ as in~\eqref{eq:25},}
    \\ &|\NL'(z)| \le C (1+|z|^\gamma) \quad \text{for a.e. } z \in \R,
\end{align*}
with $\gamma > 0$ if $d=1,2$, or $0 \le \gamma \le 2$ if~$d=3$, existence and uniqueness of a strong solution of~\eqref{eq:23} was proved in~\cite[Theorem 3.7]{Barbu1993}, for~$c=1$, $F \in H^1(0,T;L^2(\Omega))$, $U_0 \in H_0^1(\Omega) \cap H^2(\Omega)$, and~$V_0 \in$ $H_0^1(\Omega)$.
This covers the nonlinear Klein-Gordon equation for $\rho > 0$ if~$d=1,2$, or $0 < \rho \leq 2$ if~$d=3$. 
\end{remark}
\noindent 
With~$G$ as in~\eqref{eq:25}, problem~\eqref{eq:24} is a non-autonomous Hamiltonian system with Hamiltonian 
\begin{equation} \label{eq:26}
    \mathcal{H}(V,U,t) := \frac{1}{2}\Big(\int_\Omega |V(\bx,t)|^2 \dd \bx + \int_\Omega |\nabla_{\bx} U(\bx,t)|^2 \dd \bx \Big) + \int_\Omega G\left(U(\bx,t)\right) \dd \bx   + \int_{\Omega}F(\bx, t)\,U(\bx,t)\dd \bx,
\end{equation}
see e.g.~\cite[\S3.2]{MaRa1999}. Since the system is non-autonomous, the Hamiltonian (or energy) is generally not conserved, whereas the symplectic form is conserved; see~\cite[Theorem 2.1]{SanzSernaCalvo1994}. However, if $F$ is independent of $t$, the system is actually autonomous and~$\frac{d}{dt}\mathcal{H}=0$. 

\subsection{Discrete formulations}\label{sec:32}

The method introduced in \cite{BalesLasiecka1994, FrenchPeterson1996} for the linear wave equation (see~\eqref{eq:10}) was extended in~\cite{KarakashianMakridakis2005} to semilinear wave equations. 
This extension reads as follows: 

%%%%%%%%%%%%%%%%%%%%%%%%%%%%%%%%%%%%
\begin{tcolorbox}[
    colframe=black!50!white,
    colback=blue!5!white,
    boxrule=0.5mm,
    sharp corners,
    boxsep=0.5mm,
    top=0.5mm,
    bottom=0.5mm,
    right=1mm,
    left=1mm
]
    \begingroup
    \setlength{\abovedisplayskip}{0pt}
    \setlength{\belowdisplayskip}{0pt}
find $\widetilde{U}_{\bh}^{p_t}, \widetilde{V}_{\bh}^{p_t} \in Q_{\bh}^{p_t}(Q_T)$ such that $\widetilde{U}_{\bh}^{p_t}(\cdot,0) = U_{0,h_{\bx}}$, $ \widetilde{V}_{\bh}^{p_t}(\cdot,0) = V_{0,h_{\bx}}$ in~$\Omega$, and \medskip
\begin{equation} \label{eq:27}
	\begin{cases}
		(\partial_t \widetilde{U}_{\bh}^{p_t}, \chi_{\bh}^{p_t-1})_{L^2(Q_T)} - (\widetilde{V}_{\bh}^{p_t}, \chi_{\bh}^{p_t-1})_{L^2(Q_T)}=0 & \text{for all~} \chi_{\bh}^{p_t-1} \in \partial_t Q_{\bh}^{p_t}(Q_T), \\
		(\partial_t \widetilde{V}_{\bh}^{p_t}, \lambda_{\bh}^{p_t-1})_{L^2(Q_T)} + (c^2 \nabla_{\bx} \widetilde{U}_{\bh}^{p_t}, \nabla_{\bx} \lambda_{\bh}^{p_t-1})_{L^2(Q_T)} &
        \\ \hspace{1.3cm}  + \, (\NL(\widetilde{U}_{\bh}^{p_t}),\lambda_{\bh}^{p_t-1})_{L^2(Q_T)} = (F, \lambda_{\bh}^{p_t-1})_{L^2(Q_T)} & \text{for all~} \lambda_{\bh}^{p_t-1} \in \partial_t Q_{\bh}^{p_t}(Q_T).
	\end{cases}
\end{equation}
    \endgroup
\end{tcolorbox}
%%%%%%%%%%%%%%%%%%%%%%%%%%%%%%%%

\smallskip
\noindent
By proceeding as in Section~\ref{sec:24}, one proves that the method in~\eqref{eq:27} can be equivalently reformulated as the following second-order-in-time discrete formulation:
\begin{tcolorbox}[
    colframe=black!50!white,
    colback=blue!5!white,
    boxrule=0.5mm,
    sharp corners,
    boxsep=0.5mm,
    top=0.5mm,
    bottom=0.5mm,
    right=0.25mm,
    left=0.1mm
]
    \begingroup
    \setlength{\abovedisplayskip}{0pt}
    \setlength{\belowdisplayskip}{0pt}
\, find $U_{\bh}^{p_t} \in Q^{p_t}_{\bh}(Q_T)$ such that $U_{\bh}^{p_t}(\cdot,0) = U_{0,h_{\bx}}$ in~$\Omega$ and\medskip
\begin{equation} \label{eq:28}
\begin{aligned}
	-(& \partial_t U_{\bh}^{p_t},  \partial_t W_{\bh}^{p_t})_{L^2(Q_T)} + (c^2 \nabla_{\bx} U_{\bh}^{p_t} , \nabla_{\bx} \Pi_{h_t}^{(p_t-1,-1)} W_{\bh}^{p_t})_{L^2(Q_T)} 
    \\ & + (\NL(U_{\bh}^{p_t}), \Pi_{h_t}^{(p_t-1,-1)} W_{\bh}^{p_t})_{L^2(Q_T)} 
     = (F,\Pi_{h_t}^{(p_t-1,-1)} W_{\bh}^{p_t})_{L^2(Q_T)} + (V_{0,h_{\bx}},W_{\bh}^{p_t}(\cdot,0))_{L^2(\Omega)}
\end{aligned}
\end{equation}

\medskip\noindent
for all $W_{\bh}^{p_t} \in Q^{p_t}_{\bh,\bullet,0}(Q_T)$. 
    \endgroup
\end{tcolorbox}
\noindent
We emphasize that, in this case, the projection operator~$\Pi^{(p_t-1,-1)}_{h_t}$ in front of the test function is also inserted in the nonlinear term.
\begin{remark}
The coincidence of the $L^2$ projection and the elementwise Lagrange interpolation operator at the Gauss--Legendre nodes for functions in~$S_{h_t}^{p_t}(0,T)$, established in~\eqref{eq:9} and formalized in \eqref{eq:31} below, is particularly useful in the case of formulation~\eqref{eq:28}. This reduces the computation of the $L^2$ projections, including the one within the nonlinear term, to simple interpolation, thereby avoiding the solution of local linear systems, a key advantage previously discussed in Section~\ref{sec:23}.
\end{remark}
\noindent
We state the following equivalence result, whose proof is a straightforward extension of that of Theorem \ref{thm:25}, and thus is not reported.
\begin{proposition} \label{prop:34}
If~$(\widetilde{U}_{\bh}^{p_t},\widetilde{V}_{\bh}^{p_t})\in Q_{\bh}^{p_t}(Q_T)\times Q_{\bh}^{p_t}(Q_T)$ is a solution of~\eqref{eq:27}, then~$U_{\bh}^{p_t} = \widetilde{U}_{\bh}^{p_t}$ is a solution of~\eqref{eq:28}. Conversely, if~$U_{\bh}^{p_t} \in Q_{\bh}^{p_t}(Q_T)$ is a solution of~\eqref{eq:28}, then the pair~$(\widetilde{U}_{\bh}^{p_t},\widetilde{V}_{\bh}^{p_t})\in Q_{\bh}^{p_t}(Q_T)\times Q_{\bh}^{p_t}(Q_T)$, with
\begin{equation*}
    \widetilde{U}_{\bh}^{p_t}=U_{\bh}^{p_t}, \quad \widetilde{V}_{\bh}^{p_t} \quad \text{~such that} \quad \Pi_{h_t}^{(p_t-1,-1)} \widetilde{V}_{\bh}^{p_t} = \partial_t U_{\bh}^{p_t},  \quad \widetilde{V}_{\bh}^{p_t}(\cdot,0) = V_{0,h_{\bx}}
\end{equation*}
is a solution of~\eqref{eq:27}. 
\end{proposition}

\noindent
For the formulation in~\eqref{eq:27}, the existence and uniqueness of discrete solutions have been established in the following cases:
\begin{itemize}
\item globally Lipschitz function $g$: in this case, $h_t$ is required to be sufficiently small (bounded by a constant independent of $h_{\bx}$), see~\cite[Theorem~2.1]{KarakashianMakridakis2005};
\item locally Lipschitz function $g$: in this case, $h_{\bx}$ and $h_t$ are required to be sufficiently (bounded by constants independent of $h_{\bx}$ and $h_t$); this follows by combining the results in~\cite[\S~2]{KarakashianMakridakis2005}, \cite[\S~2]{KarakashianMakridakis1999}, and~\cite[\S~2 and \S 4]{KarakashianMakridakis1998}.
\end{itemize}

\noindent
Additionally, assuming that the spatial domain~$\Omega$ satisfies the elliptic regularity property, that~$V_0 \in H^1_0(\Omega)$, and that the elliptic projection is used also to define~$V_{0,h_{\bx}}$ from~$V_0$, the following error estimates were derived in~\cite[Theorem~4.1]{KarakashianMakridakis2005}, under the assumption of existence of discrete solutions and sufficient regularity of the continuous solution~$U$:
\begin{equation} \label{eq:semilin_estimKM}
    \| U - \widetilde{U}_{\bh}^{p_t} \|_{C([0,T];L^2(\Omega))} + \| \partial_t U - \widetilde{V}_{\bh}^{p_t} \|_{C([0,T];L^2(\Omega))} \apprle  h_{\bx}^{p_{\bx}+1} + h_t^{p_t+1},
\end{equation}
where the implied constant is independent of~$h_t$ and~$h_{\bx}$.
Finally, we have the following energy conservation property.

\begin{lemma}
Suppose $F = 0$. Let $\widetilde{U}_{\bh}^{p_t}, \widetilde{V}_{\bh}^{p_t}$ be a solution of \eqref{eq:27}. Then, for all~$j=0,\ldots, N_t$, 
\begin{equation*}
\begin{aligned}
   & \frac{1}{2} \left( \| \widetilde{V}_{\bh}^{p_t}(\cdot,t_j)\|^2_{L^2(\Omega)} + \| c\nabla_{\bx} \widetilde{U}_{\bh}^{p_t}(\cdot,t_j)\|^2_{L^2(\Omega)} \right) + \| G(\widetilde{U}_{\bh}^{p_t}(\cdot,t_j))\|^2_{L^2(\Omega)}
   \\ & \hspace{2cm} = \frac{1}{2} \left( \| V_{0,h_{\bx}}
   \|^2_{L^2(\Omega)} + \| c\nabla_{\bx} U_{0,h_{\bx}}
   \|^2_{L^2(\Omega)} \right) + \| G(U_{0,h_{\bx}})
   \|^2_{L^2(\Omega)},
\end{aligned}
\end{equation*}
with~$G$ as in~\eqref{eq:25}.
\end{lemma}
\begin{proof}
As in the linear case, consider testing \eqref{eq:27}, for $j=0,\ldots,N_t$, with 
\begin{equation*}
\chi_{\bh}^{p_t-1}(\bx,t) = -\partial_t \widetilde{V}_{\bh}^{p_t}(\bx,t) \mathds{1}_{[0,t_j]}(t) \quad \text{and} \quad \lambda_{\bh}^{p_t-1}(\bx,t) = \partial_t \widetilde{U}_{\bh}^{p_t}(\bx,t) \mathds{1}_{[0,t_j]}(t),
\end{equation*}
and adding the two equations. This gives the relation
\begin{equation*}
    (\widetilde{V}_{\bh}^{p_t}, \partial_t \widetilde{V}_{\bh}^{p_t})_{L^2(\Omega \times [0,t_j])} + (c^2 \nabla_{\bx} \widetilde{U}_{\bh}^{p_t}, \partial_t \widetilde{U}_{\bh}^{p_t})_{L^2(\Omega \times [0,t_j])} + (\NL(\widetilde{U}_{\bh}^{p_t}), \partial_t \widetilde{U}_{\bh}^{p_t})_{L^2(\Omega \times [0,t_j])}=0.
\end{equation*}
The result readily follows from~$\NL(U) = G'(U)$.
\end{proof}

\noindent
As in the linear case, the equivalence result of Proposition~\eqref{prop:34} allows the properties of~\eqref{eq:27} to be carried over to~\eqref{eq:28}. In particular, for the error estimates, which hold under the same assumptions as for the formulation in~\eqref{eq:27}, we have
\begin{equation} \label{eq:31b}
\begin{aligned} 
    \| U - U_{\bh}^{p_t} \|_{C([0,T];L^2(\Omega))} & \apprle h_{\bx}^{p_{\bx}+1} + h_t^{p_t+1} ,\\ 
    \| \partial_t U - \widetilde{V}_{\bh}^{p_t} \|_{C([0,T];L^2(\Omega))} & \apprle h_{\bx}^{p_{\bx}+1} + h_t^{p_t+1},
\end{aligned}
\end{equation}
where $\widetilde{V}_{\bh}^{p_t}\in Q_{\bh}^{p_t}(Q_T)$ is the reconstruction of~$\partial_t U_{\bh}^{p_t}$ defined in~\eqref{eq:22}. Without reconstructing~$\partial_t U_{\bh}^{p_t}$, we have
\begin{equation} \label{eq:32b}
    \| \partial_t U - \partial_t U_{\bh}^{p_t} \|_{L^2(Q_T)} \lesssim h_{\bx}^{p_{\bx}+1} + h_t^{p_t}.
\end{equation}
Finally, for~$F=0$, the energy conservation property reads, for $j=0,\ldots,N_t$,
\begin{equation*}
\begin{aligned}
   & \frac{1}{2} \left( \| \widetilde{V}_{\bh}^{p_t}(\cdot,t_j)\|^2_{L^2(\Omega)} + \| c\nabla_{\bx} {U}_{\bh}^{p_t}(\cdot,t_j)\|^2_{L^2(\Omega)} \right) + \| G(U_{\bh}^{p_t}(\cdot,t_j))\|^2_{L^2(\Omega)}
   \\ & \hspace{2cm} = \frac{1}{2} \left( \| V_{0,h_{\bx}} \|^2_{L^2(\Omega)} + \| c\nabla_{\bx} U_{0,h_{\bx}} \|^2_{L^2(\Omega)} \right) + \| G(U_{0,h_{\bx}}) \|^2_{L^2(\Omega)},
\end{aligned}
\end{equation*}

\subsubsection{Numerical results} \label{sec:321}

\noindent
In this section, we verify the convergence estimate in \eqref{eq:semilin_estimKM} for the sine-Gordon equation. We consider problem~\eqref{eq:24} with $g(U)= \sin(U)$ on $Q_T = \Omega \times (0,T)$, with $\Omega = (-20,20)$ and $T=1$, and data
\begin{equation*}
    c(x)=1, \qquad F(x,t) = 0, \qquad U_0(x) = 0, \qquad V_0(x) = \frac{4}{\gamma} \sech (\frac{x}{\gamma}), \quad \gamma = 1.1.
\end{equation*}
The exact solution of \eqref{eq:24} is
\begin{equation} \label{eq:33a}
    U(x,t) = 4\arctan (\phi(t) \sech(\frac{x}{\gamma})), \quad\quad 
        \phi(t) = \frac{1}{\sqrt{\gamma^2-1}}  \sin(\frac{t}{\gamma} \sqrt{\gamma^2-1}).
\end{equation}
Since
\begin{equation*}
    \max_{t \in (0,1), \,\, x \in \{-20,20\}} |U(x,t)| \le 10^{-16},
\end{equation*}
we can assume homogeneous Dirichlet boundary conditions. We test the formulation in~\eqref{eq:28} on uniform meshes for varying mesh sizes $h_t = h_{\bx}$, using $p_t = p_{\bx} = 1$ and $p_t = p_{\bx} = 2$. The nonlinear systems are solved using a fixed-point iteration. In Figure~\ref{fig:2}, we report the absolute errors in the $C^0([0,T];L^2(\Omega))$ and $H^1(0,T;L^2(\Omega))$ norms for $U - U_{\bh}^{p_t}$, as well as in the 
$C^0([0,T];L^2(\Omega))$ norm for $\partial_t U - \widetilde{V}_{\bh}^{p_t}$. The observed convergence behavior agrees with the estimates~\eqref{eq:31b} and~\eqref{eq:32b}.
\begin{figure}[h!]
    \centering
    \begin{subfigure}[b]{0.495\textwidth}
    \centering
         \includegraphics[width=\textwidth]{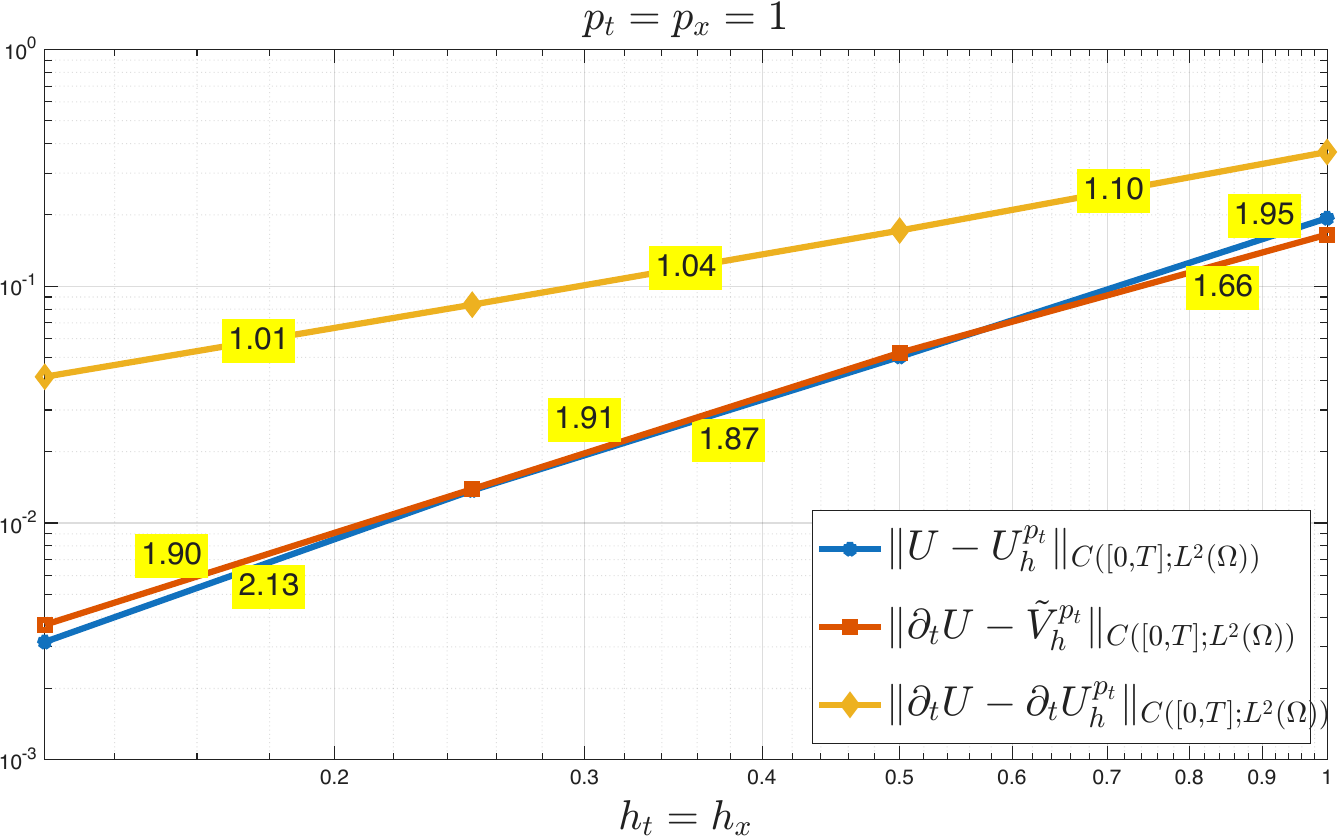}
     \end{subfigure}
    \begin{subfigure}[b]{0.495\textwidth}
    \centering
         \includegraphics[width=\textwidth]{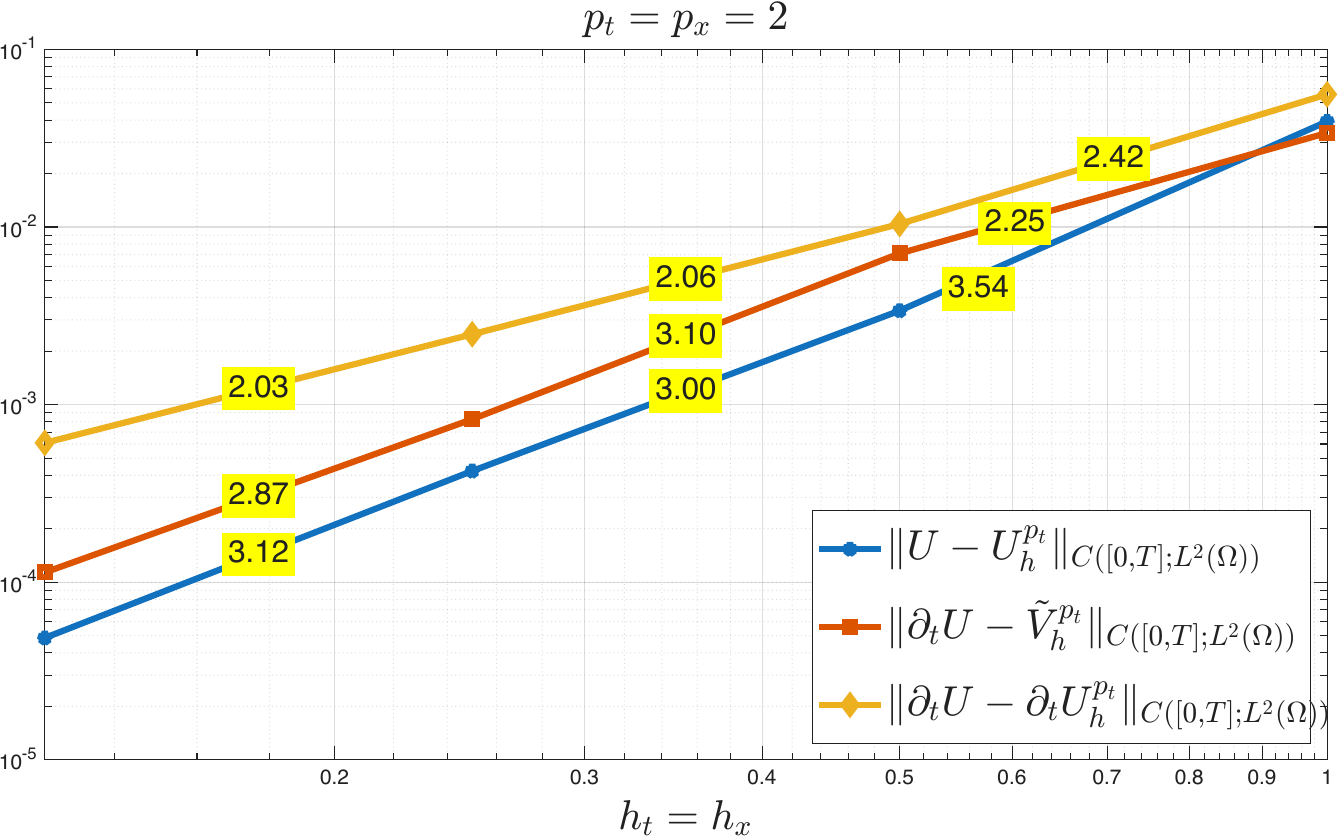}
     \end{subfigure}
\caption{Absolute errors with respect to the exact solution~\eqref{eq:33a} for varying mesh sizes $h_t=h_x$, using polynomial degrees \(p_t=p_x=1\) (left) and \(p_t=p_x=2\) (right).}
\label{fig:2}
\end{figure}

\subsection{Symplectic variants based on quadrature in time} \label{sec:33}
The second-order-in-time discretization~\eqref{eq:28} provides unconditional stability and exact conservation of a discrete energy.
However, the integration of nonlinear terms requires a quadrature rule, and therefore conservation of energy in practice holds up to quadrature error. Therefore, if one wants to achieve conservation of energy up to machine precision, the quadrature rule should be selected accordingly, making the scheme problem-dependent. For example, in~\cite{FrenchSchaeffer90}, French and Sch\"{a}ffer propose a dynamic quadrature rule. Another route is to fix the quadrature rule once and for all, therefore giving up energy conservation, possibly in favor of other geometric properties. Interesting choices are the Gauss--Legendre and Gauss--Lobatto quadratures, since they make the scheme symplectic in time, as we show in the following.

\smallskip
\noindent 
We emphasize that, in this section, we adopt a slightly different point of view, departing from space-time schemes. In fact, the methods considered in this section can all be understood as a semi-discretization in space $S^{p_{\bx}}_{h_{\bx}}(\Omega)$ followed by a discretization in time with a Runge--Kutta scheme. The space semi-discretization has the important property of retaining the Hamiltonian structure of the wave equation, see~\cite{SanchezValenzuela25}.
Therefore, it make sense to look for symplectic integrators in time. In particular, when the system is non-autonomous, the energy is not conserved even at the continuous level, whereas the symplectic form is. Therefore, symplectic schemes are preferable to energy-conserving ones in this context.

\begin{remark}
Equivalence between Galerkin schemes for ordinary differential equations (ODEs) and Runge–Kutta methods has been explored in various contexts. In \cite{Hulme72a,Hulme72b}, a one-step Galerkin method is proposed, in which the trial function is continuous, while the test function is discontinuous but of the same polynomial degree. The author proves the equivalence of this formulation with a Runge–Kutta collocation method. In \cite{FrenchSchaeffer90}, continuous–discontinuous Galerkin methods for general ODEs are derived, and it is shown that the variant employing Lagrange interpolation at Gauss–Legendre points is equivalent to an implicit Runge–Kutta method. Finally, \cite{Bottasso97} establishes that a general class of finite element time-integration schemes, when combined with an appropriate quadrature rule, is equivalent to a Runge–Kutta method; in particular, the use of Gauss–Legendre quadrature leads to the Gauss Runge–Kutta scheme.
\end{remark}
\noindent
We use the notation $(\cdot,\cdot)_{L^2_{\GLe,m_1}(Q_T)}$ and $(\cdot,\cdot)_{L^2_{\GLo,m_2}(Q_T)}$ to denote the $L^2$ inner products over $Q_T$ computed by means of quadrature rules in time with $m_1 \ge 1$ Gauss--Legendre and $m_2 \ge 2$ Gauss--Lobatto nodes, respectively, in each time slab $[t_{n-1},t_n]$. We also recall that Gauss--Legendre quadrature formulas are exact for polynomials of degree up to $2m_1 - 1$, whereas Gauss--Lobatto quadrature formulas are exact for polynomials of degree up to $2m_2 - 3$.

\smallskip
\noindent
Assuming~$F\in C([0,T];L^2(\Omega))$, we introduce the following two variants of \eqref{eq:28}:

%%%%%%%%%%%%%%%%%%%%%%%%%%%%%%%%%%%%%%%%%%%%%%%%%%%%%%%
\begin{tcolorbox}[
    colframe=black!50!white,
    colback=blue!5!white,
    boxrule=0.5mm,
    sharp corners,
    boxsep=0.5mm,
    top=0.5mm,
    bottom=0.5mm,
    right=0.25mm,
    left=0.1mm
]
    \begingroup
    \setlength{\abovedisplayskip}{0pt}
    \setlength{\belowdisplayskip}{0pt}
\, find $U_{\bh}^{p_t} \in Q^{p_t}_{\bh}(Q_T)$ such that $U_{\bh}^{p_t}(\cdot,0) = U_{0,h_{\bx}}$ in~$\Omega$ and\medskip
\begin{equation} \label{eq:29}
\begin{aligned}
	-(& \partial_t U_{\bh}^{p_t},  \partial_t W_{\bh}^{p_t})_{L^2(Q_T)} + (c^2 \nabla_{\bx} U_{\bh}^{p_t} , \nabla_{\bx} W_{\bh}^{p_t})_{L_{\GLe,p_t}^2(Q_T)} 
    \\ & + (\NL(U_{\bh}^{p_t}), W_{\bh}^{p_t})_{L_{\GLe,p_t}^2(Q_T)} 
     = (F,W_{\bh}^{p_t})_{L_{\GLe,p_t}^2(Q_T)} + (V_{0,h_{\bx}},W_{\bh}^{p_t}(\cdot,0))_{L^2(\Omega)}
\end{aligned}
\end{equation}

\medskip\noindent
\, for all $W_{\bh}^{p_t} \in Q^{p_t}_{\bh,\bullet,0}(Q_T)$;
    \endgroup
\end{tcolorbox}
%%%%%%%%%%%%%%%%%%%%%%%%%%%%%%%%%%%%%%%%%%%%%%%%%%%%%%%

%%%%%%%%%%%%%%%%%%%%%%%%%%%%%%%%%%%%%%%%%%%%%%%%%%%%%%%
\begin{tcolorbox}[
    colframe=black!50!white,
    colback=blue!5!white,
    boxrule=0.5mm,
    sharp corners,
    boxsep=0.5mm,
    top=0.5mm,
    bottom=0.5mm,
    right=0.25mm,
    left=0.1mm
]
    \begingroup
    \setlength{\abovedisplayskip}{0pt}
    \setlength{\belowdisplayskip}{0pt}
\, find $U_{\bh}^{p_t} \in Q^{p_t}_{\bh}(Q_T)$ such that $U_{\bh}^{p_t}(\cdot,0) = U_{0,h_{\bx}}$ in~$\Omega$ and\medskip
\begin{equation} \label{eq:30}
\begin{aligned}
	-(& \partial_t U_{\bh}^{p_t}, \partial_t W_{\bh}^{p_t})_{L^2(Q_T)} + (c^2 \nabla_{\bx} U_{\bh}^{p_t}, \nabla_{\bx}  W_{\bh}^{p_t})_{L_{\GLo,p_t+1}^2(Q_T)} 
    \\ & + (\NL(U_{\bh}^{p_t}), W_{\bh}^{p_t})_{L_{\GLo,p_t+1}^2(Q_T)}
     = (F,W_{\bh}^{p_t})_{L_{\GLo,p_t+1}^2(Q_T)} + (V_{0,h_{\bx}},W_{\bh}^{p_t}(\cdot,0))_{L^2(\Omega)}
\end{aligned}
\end{equation}

\medskip\noindent
\, for all $W_{\bh}^{p_t} \in Q^{p_t}_{\bh,\bullet,0}(Q_T)$.
    \endgroup
\end{tcolorbox}
%%%%%%%%%%%%%%%%%%%%%%%%%%%%%%%%%%%%%%%%%%%%%%%%%%%%%%%

\smallskip

\begin{remark} \label{rem:37}
Let~$\Int: C([0,T])\to \partial_t S_{h_t}^{p_t}(0,T)$ be the Lagrange interpolator defined, element-by-element, at the~$p_t$ Gauss--Legendre nodes. Analogously to the $L^2$ projector~$\Pi_{h_t}^{(p_t-1,-1)}$, we denote by the same symbol~$\Int$ both this operator and its space--time extension, acting on functions in~$C([0,T];L^2(\Omega))$. For all $v_{h_t}^{p_t} \in S_{h_t}^{p_t}(0,T)$,
\begin{equation} \label{eq:31}
    \Pi_{h_t}^{(p_t-1,-1)} v_{h_t}^{p_t}
    = \Int v_{h_t}^{p_t};
\end{equation}
see~\eqref{eq:9}. Using the operator~$\Int$, the scheme in~\eqref{eq:29} can be written equivalently in the following form:
find $U_{\bh}^{p_t} \in Q^{p_t}_{\bh}(Q_T)$ such that $U_{\bh}^{p_t}(\cdot,0) = U_{0,h_{\bx}}$ in~$\Omega$ and  \medskip
\begin{equation} \label{eq:32}
\begin{aligned}
	-(& \partial_t U_{\bh}^{p_t},  \partial_t W_{\bh}^{p_t})_{L^2(Q_T)} + (\Int c^2 \nabla_{\bx} U_{\bh}^{p_t}, \nabla_{\bx} W_{\bh}^{p_t})_{L^2(Q_T)} 
    \\ & + (\Int \NL(U_{\bh}^{p_t}), W_{\bh}^{p_t})_{L^2(Q_T)} 
     = (\Int F,W_{\bh}^{p_t})_{L^2(Q_T)} + (V_{0,h_{\bx}},W_{\bh}^{p_t}(\cdot,0))_{L^2(\Omega)}
\end{aligned}
\end{equation}
Then, from the reformulation~\eqref{eq:32} of~\eqref{eq:29}, it is clear that the methods in~\eqref{eq:28} and~\eqref{eq:29} differ only by the presence of~$\Int$ in the nonlinear and source terms (the projector~$\Pi_{h_t}^{(p_t-1,-1)}$ applied to the test functions in these terms in~\eqref{eq:29} is redundant and is therefore omitted). Equivalently, taking into account~\eqref{eq:31}, this difference can also be interpreted as interpolating the first argument of these terms rather than the test functions.
\end{remark}

\smallskip
\noindent
In the following theorem, we show that the formulations in~\eqref{eq:29} and~\eqref{eq:30} correspond to suitable collocation Runge–Kutta time discretizations, which in turn implies their symplecticity.

\begin{theorem}\label{th:GLGL}
Schemes~\eqref{eq:29} and~\eqref{eq:30} are equivalent, respectively, 
to the Gauss--Legendre 
and the pair Lobatto IIIA/IIIB Runge--Kutta time discretizations applied to the first-order-in-time formulation~\eqref{eq:24} of the semilinear wave equation, semi-discretized in space with $S_{h_{\bx}}^{p_{\bx}}(\Omega)$ spaces. Therefore, both schemes are symplectic.
\end{theorem}
\begin{proof}
Proceeding as in Section~\ref{sec:24}, it can be proved that the method in~\eqref{eq:29} is equivalent (in the sense of Proposition~\ref{prop:34}) to the following variant of~\eqref{eq:27}: find $\widetilde{U}_{\bh}^{p_t}, \widetilde{V}_{\bh}^{p_t} \in Q_{\bh}^{p_t}(Q_T)$ such that $\widetilde{U}_{\bh}^{p_t}(\cdot,0) = U_{0,h_{\bx}}$, $ \widetilde{V}_{\bh}^{p_t}(\cdot,0) = V_{0,h_{\bx}}$ in~$\Omega$, and
\begin{equation} \label{eq:33}
	\begin{cases}
		(\partial_t \widetilde{U}_{\bh}^{p_t}, \chi_{\bh}^{p_t-1})_{L^2(Q_T)} - (\widetilde{V}_{\bh}^{p_t}, \chi_{\bh}^{p_t-1})_{L^2(Q_T)}=0 & \text{for all~} \chi_{\bh}^{p_t-1} \in \partial_t Q_{\bh}^{p_t}(Q_T), \vspace{0.05cm} \\
		(\partial_t \widetilde{V}_{\bh}^{p_t}, \lambda_{\bh}^{p_t-1})_{L^2(Q_T)} + (c^2 \nabla_{\bx} \widetilde{U}_{\bh}^{p_t}, \nabla_{\bx} \lambda_{\bh}^{p_t-1})_{L^2(Q_T)} &
        \vspace{0.05cm} \\ \hspace{1.3cm}  + \, (\NL(\widetilde{U}_{\bh}^{p_t}), \lambda_{\bh}^{p_t-1})_{L_{\GLe,p_t}^2(Q_T)} = (F, \lambda_{\bh}^{p_t-1})_{L_{\GLe,p_t}^2(Q_T)} & \text{for all~} \lambda_{\bh}^{p_t-1} \in \partial_t Q_{\bh,\bullet,0}^{p_t}(Q_T).
	\end{cases}
\end{equation}
It is well-known (see, e.g. \cite{FrenchSchaeffer90}) that scheme \eqref{eq:33} is equivalent to a Gauss--Legendre Runge--Kutta discretization in time, and Gauss--Legendre Runge--Kutta time discretizations are known to be symplectic (see \cite{Lasagni88,SanzSerna88}).
For completeness, we include the direct proof of  the equivalence in the Appendix.

\smallskip
\noindent 
We consider now the method in~\eqref{eq:30}. As in Lemma \ref{lem:23}, an equivalent formulation of~\eqref{eq:30} is the following: find $U_{\bh}^{p_t} \in Q_{\bh}^{p_t}(Q_T)$ and~$V_{\bh}^{p_t-1} \in \partial_t Q_{\bh}^{p_t}(Q_T)$ such that $U_{\bh}^{p_t}(\cdot,0) = U_{0,h_{\bx}}$ in~$\Omega$ and
\begin{equation} \label{eq:34}
	\begin{cases}
		(\partial_t U_{\bh}^{p_t}, \chi_{\bh}^{p_t-1})_{L^2(Q_T)} - (V_{\bh}^{p_t-1}, \chi_{\bh}^{p_t-1})_{L^2(Q_T)} = 0 & \text{for all~} \chi_{\bh}^{p_t-1} \in \partial_t Q_{\bh}^{p_t}(Q_T), \vspace{0.05cm} \\
		-(V_{\bh}^{p_t-1}, \partial_t W_{\bh}^{p_t})_{L^2(Q_T)} + (c^2 \nabla_{\bx} U_{\bh}^{p_t}, \nabla_{\bx} W_{\bh}^{p_t})_{L_{\GLo,p_t+1}^2(Q_T)} &
        \vspace{0.05cm} \\ \hspace{1.3cm}  + \, (\NL(U_{\bh}^{p_t}), W_{\bh}^{p_t})_{L_{\GLo,p_t+1}^2(Q_T)} = (F, W_{\bh}^{p_t})_{L_{\GLo,p_t+1}^2(Q_T)} &
        \\ \hspace{5cm}+ (V_{0,h_{\bx}},W_{\bh}^{p_t}(\cdot,0))_{L^2(\Omega)} & \text{for all~} W_{\bh}^{p_t} \in Q_{\bh,\bullet,0}^{p_t}(Q_T).
	\end{cases}
\end{equation}
To prove that \eqref{eq:34} is actually a Lobatto IIIA/IIIB Runge--Kutta method, we consider the following hybridized-in-time formulation: find $U_{\bh}^{p_t} \in Q_{\bh}^{p_t}(Q_T)$, $V_{\bh}^{p_t-1} \in \partial_t Q_{\bh}^{p_t}(Q_T)$, and $\widehat{V}_{h_{\bx}}^n \in S_{h_{\bx}}^{p_{\bx}}$ for $n=0,\ldots,N_t$ such that $\widehat{V}_{h_{\bx}}^0 = V_{0,h_{\bx}}$, $U_{\bh}^{p_t}(\cdot,0)= U_{0,h_{\bx}}$ in $\Omega$, and 
\begin{equation} \label{eq:35}
	\begin{cases}
		(\partial_t U_{\bh}^{p_t}, \chi_{\bh}^{p_t-1})_{L^2(Q_T)} - (V_{\bh}^{p_t-1}, \chi_{\bh}^{p_t-1})_{L^2(Q_T)}=0 & \text{for all~} \chi_{\bh}^{p_t-1} \in \partial_t Q_{\bh}^{p_t}(Q_T), \vspace{0.05cm}
        \\ (\partial_t V_{\bh}^{p_t-1}, \lambda_{h_{\bx},n}^{p_t})_{L^2(\Omega \times (t_{n-1},t_{n}))} &
        \\ \hspace{1cm} + (\widehat{V}_{h_{\bx}}^n - V^{p_t-1}_{\bh}(\cdot, t_{n}^{-}),\lambda_{h_{\bx},n}^{p_t}(\cdot,t_{n}))_{L^2(\Omega)} & \vspace{0.05cm}
        \\ \hspace{1cm} - (\widehat{V}_{h_{\bx}}^{n-1}  - V^{p_t - 1}_{\bh}(\cdot, t_{n-1}^+), \lambda_{h_{\bx},n}^{p_t}(\cdot,t_{n-1}))_{L^2(\Omega)} &
        \\ \hspace{1cm} + (c^2 \nabla_{\bx} U_{\bh}^{p_t}, \nabla_{\bx} \lambda_{h_{\bx},n}^{p_t})_{L_{\GLo,p_t+1}^2(\Omega \times (t_{n-1},t_{n}))} &
        \\ \hspace{1cm} = (F, \lambda_{h_{\bx},n}^{p_t})_{L_{\GLo,p_t+1}^2(\Omega \times (t_{n-1},t_{n}))} & \text{for all~} \lambda_{h_{\bx},n}^{p_t} \in S_{h_{\bx}}^{p_{\bx}}(\Omega) \times \mathbb{P}^{p_t}(t_{n-1},t_{n}),
        \\ \phantom{\hspace{1cm} = (F, \lambda_{h_{\bx},n}^{p_t})_{L_{\GLo,p_t+1}^2(\Omega \times (t_{n-1},t_{n}))}} & \text{for all~}n = 1,\ldots,N_t.
	\end{cases}
\end{equation}
The local test functions~$\lambda_{h_{\bx},n}^{p_t}$ are restrictions to each time slab of global functions~$\lambda_{\bh}^{p_t}\in \partial_t Q^{p_t+1}_{\bh}(Q_T)$. If $U_{\bh}^{p_t}$, $V^{p_t - 1}_{\bh}$, and $\widehat{V}_{h_{\bx}}^n$, $n=0,\ldots,N_t,$
solve the hybridized formulation~\eqref{eq:35}, then $U_{\bh}^{p_t}$ and $V^{p_t - 1}_{\bh}$ solve~\eqref{eq:34}. This can be easily seen by testing the block of equations forming the second part of~\eqref{eq:35} with~$\lambda_{\bh}^{p_t}=W_{\bh}^{p_t}\in Q^{p_t}_{\bh, \bullet, 0}(Q_T)$, which is possible since $Q^{p_t}_{\bh, \bullet, 0}(Q_T) \subset \partial_t Q^{p_t+1}_{\bh}(Q_T)$. Viceversa, let $U_{\bh}^{p_t}$ and $V^{p_t - 1}_{\bh} $ solve~\eqref{eq:34}. Then, $\{\widehat{V}_{h_{\bx}}^n\}_n$ can be computed sequentially as follows: given $\widehat{V}_{h_{\bx}}^{n-1}$, compute $\widehat{V}_{h_{\bx}}^n$ by solving the linear system associated with a mass matrix in space, which is obtained isolating all terms not involving $\widehat{V}_{h_{\bx}}^n$ in \eqref{eq:35}. It follows that $U_{\bh}^{p_t}$, $V^{p_t - 1}_{\bh}$ and the $\widehat{V}_{h_{\bx}}^n$ satisfy the the second equation in~\eqref{eq:35} by construction.
\smallskip

\noindent
We show now that \eqref{eq:35} is a Lobatto IIIA/IIIB Runge--Kutta method. We start by recalling the discontinuous collocation formulation of Lobatto IIIA/IIIB pairs in the case of a partitioned system of the type
\begin{equation*}
    \begin{cases}
       \mdot y = f(z,t),
        \\ \mdot z = g(y,t);
    \end{cases}
\end{equation*}
see \cite[Proof of Theorem 2.2, II.2.2]{HairerLubichWannerBook} for the autonomous case. In each time element $[t_{{n-1}}, t_{{n}}]$, given~$y^{{n-1}}$ and~$z^{{n-1}}$,  we look for polynomials $u^{p_t+1} \in \mathbb{P}^{p_t+1}(t_{{n-1}},t_{{n}})$ and $v^{{p_t-1}} \in\mathbb{P}^{p_t-1}(t_{{n-1}},t_{{n}})$ satisfying, with $h_t^{(n)} := t_{{n}}-t_{{n-1}}$,
\begin{equation} \label{eq:36}
\begin{aligned}
    \begin{cases}
        u^{p_t+1}(t_{{n-1}})= y^{{n-1}} &
        \\  \mdot u^{p_t+1}(t_n^i) = f(v^{{p_t-1}}(t_n^i), t^i_n) & \text{for all~} i = 0, \dots, p_t,
        \\
        y^{{n}} = u^{p_t+1}(t_{{n}}), &
        \\ v^{{p_t-1}}(t_{{n-1}}) = z^{{n-1}} - h_t^{(n)} \omega_0 (\mdot{v}^{{p_t-1}}(t_{{n-1}}) - g(u^{p_t+1}(t_{{n-1}}), t_{{n-1}})), \\
        \mdot{v}^{{p_t-1}}(t_n^i) = g(u^{p_t+1}(t_n^i), t_n^i) & \text{for all~} i = 1, \dots, p_t -1, \\
        z^{{n}} = v^{{p_t-1}}(t_{{n}}) - h_t^{(n)} \omega_{p_t}(\mdot{v}^{{p_t-1}}(t_{{n}}) - g(u^{p_t+1}(t_{{n}}), t_{{n}}) ).
    \end{cases}
\end{aligned}
\end{equation}
Here, $c_0, \dots, c_{p_t}$ and $\omega_0, \dots, \omega_{p_t}$ are the Gauss--Lobatto nodes and weights, respectively, on $[0,1]$, and $t_n^i \coloneqq t_{{n-1}} + c_i h_t^{(n)}$. 
In the special case $f(z, t) = z$, {the second equation in~\eqref{eq:36}} %the collocation condition for $\mdot u^{p_t+1}$  
implies that $\mdot u^{p_t+1}$ is a polynomial of degree $p_t - 1$, and thus $u^{p_t+1}$ is actually a polynomial of degree $p_t$. 

\smallskip
\noindent 
{Comparing~\eqref{eq:35} with~\eqref{eq:36}, $y$ corresponds to~$U_{\bh}^{p_t}$ at the mesh nodes, while $u^{p_t+1}$ (a polynomial of degree at most~$p_t$ in~$[t_{n-1},t_n]$, as observed above) corresponds to~$U_{\bh}^{p_t}$ at the interior Gauss-Lobatto nodes. In turns, $z$ represents~$\widehat{V}_{h_{\bx}}^n$, and~$v^{p_t-1}$ corresponds to~$V_{\bh}^{p_t-1}$ at the interior Gauss-Lobatto nodes.}

\smallskip
\noindent 
The first equation in \eqref{eq:35} directly implies that
\begin{equation}\label{eq:GL1}
    \partial_t U^{p_t}_{\bh} = V^{p_t - 1}_{\bh}
\end{equation}
holds pointwise, in particular at Gauss--Lobatto points. We test the second equation in \eqref{eq:35} with the functions of the Lagrange basis for polynomials of degree~$\le p_t$ associated with the~$p_t+1$ Gauss--Lobatto nodes in~$[t_n,t_{n+1}]$. Recalling that the Gauss--Lobatto quadrature with~$p_t+1$ nodes is exact for polynomials of degree~$\le 2p_t-1$, we obtain the following relations at~$t_{n-1}$, at the interior Gauss--Lobatto points, and at~$t_n$:
\begin{equation}\label{eq:GL2}
\begin{aligned}
    & {h_t^{(n)}} \, \omega_0(\partial_t V^{p_t-1}_{\bh}(\cdot, t_{{n-1}}), \lambda_{h_{\bx,n}}^{p_t}(\cdot,t_{{n-1}}))_{L^2(\Omega)} - (\widehat{V}_{h_{\bx}}^{n-1} - V^{p_t-1}_{\bh}(\cdot, t_{{n-1}}), \lambda_{h_{\bx,n}}^{p_t}(\cdot,t_{{n-1}}))_{L^2(\Omega)} 
    \\ & \hspace{1cm} + {h_t^{(n)}}\omega_0(c^2\nabla_{\bx} {U}^{p_t}_{\bh}(\cdot, t_{{n-1}}), \nabla_{\bx}\lambda_{h_{\bx,n}}^{p_t}(\cdot,t_{{n-1}}))_{L^2(\Omega)}= {h_t^{(n)}}\omega_0(F(\cdot, t_{{n-1}}), \lambda_{h_{\bx,n}}^{p_t}(\cdot,t_{{n-1}}))_{L^2(\Omega)},
    \\ & {h_t^{(n)}} \omega_i(\partial_tV^{p_t-1}_{\bh}(\cdot, t^i_n), \lambda_{h_{\bx,n}}^{p_t}(\cdot,t^i_n))_{L^2(\Omega)}  + {h_t^{(n)}} \omega_i(c^2\nabla_{\bx} {U}^{p_t}_{\bh}(\cdot, t^i_n), \nabla_{\bx} \lambda_{h_{\bx,n}}^{p_t}(\cdot,t^i_n))_{L^2(\Omega)} 
    \\ & \hspace{5.5cm} = \omega_i(F(\cdot, t^i_n), \lambda_{h_{\bx,n}}^{p_t}(\cdot,t^i_n))_{L^2(\Omega)}
    {\qquad\qquad\text{for all}~i=1,\ldots,p_t-1,}
    \\ & {h_t^{(n)}} \omega_{p_t}(\partial_t V^{p_t-1}_{\bh}(\cdot, t_{{n}}), \lambda_{h_{\bx,n}}^{p_t}(\cdot,t_{{n}}))_{L^2(\Omega)} + (\widehat{V}^{n}_{h_{\bx}} - V^{p_t-1}_{\bh}(\cdot, t_{{n}}), \lambda_{h_{\bx,n}}^{p_t}(\cdot,t_{{n}}))_{L^2(\Omega)}, 
    \\ & \hspace{2.3cm} + {h_t^{(n)}} \omega_{p_t}(c^2\nabla_{\bx} \widetilde{U}^{p_t}_{\bh}(\cdot, t_{{n}}), \nabla_{\bx}\lambda_{h_{\bx,n}}^{p_t}(\cdot,t_{{n}}))_{L^2(\Omega)} 
     = {h_t^{(n)}} \omega_{p_t}(F(\cdot, t_{{n}}), \lambda_{h_{\bx,n}}^{p_t}(\cdot,t_{{n}}))_{L^2(\Omega)}.
\end{aligned}
\end{equation}
Comparing~{the reformulation in~\eqref{eq:GL1}--\eqref{eq:GL2} of~\eqref{eq:35}} %these relations 
with {the Lobatto IIIA/IIIB scheme}~\eqref{eq:36} completes the proof of the equivalence claim for~\eqref{eq:30}. Finally, the Lobatto IIIA/IIIB pair{s are} known to be a symplectic integrator{s} for autonomous~\cite{Sun93} and non-autonomous~\cite{Jay21} Hamiltonian systems.
\end{proof}
\begin{remark}
The use of hybridization to investigate the geometric properties of numerical schemes is inspired by \cite{McSt2020}.
\end{remark}

\section*{Conclusion}
We have shown that a slight modification of the stabilized formulation of~\cite{Zank2021} yields a second-order-in-time DG--CG method for the linear wave equation that is equivalent to the classical first-order-in-time approach. This equivalence allows the second-order scheme to inherit the unconditional stability and convergence properties of the first-order method while reducing the number of unknowns and enabling efficient time-stepping implementation. The framework was then extended to the semilinear wave equation. Finally, we proposed symplectic variants based on Gauss–Legendre and Gauss–Lobatto quadratures, which give up energy preservation in favor of geometric structure preservation.

\appendix

\section{First-order-in-time DG--CG time discretizations are Gauss--Legendre Runge--Kutta methods} \label{app:1}

Gauss--Legendre Runge--Kutta methods are implicit Runge--Kutta methods based on collocation at Gauss--Legendre quadrature points. For the initial-value problem~$\mdot{\boldsymbol{y}}
(t) = \boldsymbol{F}(t, \boldsymbol{y}(t))$, $\boldsymbol{y}(0) = \boldsymbol{y}_0$, the Gauss--Legendre Runge--Kutta method with~$m$ points on a time mesh $\{t_0, t_1,\ldots\}$ with step sizes $h_t^{(n)} := t_{n} - t_{n-1}$ has the following form (see, e.g. \cite[Lemma 3.5]{Iserles2009}): for every time step $n \ge 1$, we seek a polynomial $\boldsymbol{q}_n$ (with vector coefficients) of degree $m$ such that
\begin{equation} \label{eq:38}
\begin{aligned}
    \boldsymbol{q}_n(t_{n-1}) & = \boldsymbol{y}_{n-1}
    \\ \mdot{\boldsymbol{q}}_n(t_{n-1} + c_j h_t^{(n)}) & = \boldsymbol{F}(t_{n-1} + c_j h_t^{(n)}, \boldsymbol{y}(t_{n-1}+c_j h_t^{(n)})), \quad j = 1,\ldots,m,
\end{aligned}
\end{equation}
and then set $\boldsymbol{y}_n = \boldsymbol{q}_n(t_n)$. Here,~$\{c_j\}$ are the nodes of the Gauss--Legendre quadrature over $[0,1]$ with $m$ points.

\smallskip
\noindent
Consider the temporal counterpart of \eqref{eq:33}, i.e., the problem: find $u_{h_t}^{p_t}, v_{h_t}^{p_t} \in S_{h_t}^{p_t}(0,T)$ such that
\begin{equation} \label{eq:39}
	\begin{cases}
		(\mdot{u}_{h_t}^{p_t}, \chi_{h_t}^{p_t-1})_{L^2(0,T)} - (v_{h_t}^{p_t}, \chi_{h_t}^{p_t-1})_{L^2(0,T)} = 0, & \text{for all~} \chi_{h_t}^{p_t-1} \in S^{(p_t-1,-1)}_{h_t}(0,T), \\
		(\mdot{v}_{h_t}^{p_t}, \lambda_{h_t}^{p_t-1})_{L^2(0,T)} + \mu(u_{h_t}^{p_t}, \lambda_{h_t}^{p_t-1})_{L^2(0,T)}  &
        \\ \hspace{1.5cm} + (\Int g(u_{h_t}^{p_t}), \lambda_{h_t}^{p_t-1})_{L^2(0,T)} = (\Int f,\lambda_{h_t}^{p_t-1})_{L^2(0,T)} & \text{for all~} \lambda_{h_t}^{p_t-1} \in S^{(p_t-1,-1)}_{h_t}(0,T),
	\end{cases}
\end{equation}
with initial conditions $u_{h_t}^{p_t}(0) = u_0$, $v_{h_t}^{p_t}(0) = v_0$, a source $f \in C^0([0,T])$, a nonlinear term $g : \R \to \R$, and a constant~$\mu>0$. Here, $\mu$ represents an eigenvalue of the operator $-\div(c^2 \cdot)$ with Dirichlet boundary conditions.

\smallskip
\noindent
For any $n \ge 0$, consider testing in~\eqref{eq:39} with $\chi_{h_t}^{p_t-1}(t) = \lambda_{h_t}^{p_t-1}(t) = \mathcal{L}_j^{(n)}(t) \mathds{1}_{[t_n,t_{n+1}]}(t)$, where $\{\mathcal{L}_j^{(n)}(t)\}$ is the Lagrange basis for polynomials of degree~$\le p_t-1$ associated with the $p_t$ Gauss--Legendre nodes~$\{x^{(n)}_j\}$ over $[t_n,t_{n+1}]$. Note that the degree of each $\mathcal{L}_j^{(n)}(t)$ is $p_t-1$, and recall that the Gauss--Legendre quadrature rule with $p_t$ nodes is exact for polynomials up to degree $2p_t-1$. For the first equation in \eqref{eq:39}, for $j=1,\ldots,p_t$, we get
\begin{equation*}
	\begin{aligned}
		 (\mdot{u}_{h_t}^{p_t}, \mathcal{L}_j^{(n)})_{L^2(t_n,t_{n+1})} - (v_{h_t}^{p_t}, \mathcal{L}_j^{(n)})_{L^2(t_n,t_{n+1})}
		 & = \sum_{k=1}^{p_t} \mdot{u}_{h_t}^{p_t}(x_k^{(n)}) \mathcal{L}_j^{(n)}(x_k^{(n)}) \omega_k^{(n)} - \sum_{k=1}^{p_t} v_{h_t}^{p_t}(x_k^{(n)}) \mathcal{L}_j^{(n)}(x_k^{(n)}) \omega_k^{(n)}
         \\ & = \mdot{u}_{h_t}^{p_t} (x_j^{(n)}) \omega_j^{(n)} - v_{h_t}^{p_t}(x_j^{(n)}) \omega_j^{(n)},
	\end{aligned}
\end{equation*}
since $\mathcal{L}_j^{(n)}(x_k^{(n)}) = \delta_{jk}$. Here, we denoted by $\{\omega_j^{(n)}\}$ the weights of Gauss--Legendre quadrature with $p_t$ points over $[t_n,t_{n+1}]$. The first equation in~\eqref{eq:39} then implies
\begin{equation} \label{eq:40}
	\begin{aligned}
		\mdot{u}_{h_t}^{p_t} (x_j^{(n)}) - v_{h_t}^{p_t}(x_j^{(n)}) = 0, \qquad 1\le j\le p_t.
	\end{aligned}
\end{equation}
Similarly, for the terms of second equation in~\eqref{eq:39}, we obtain for, for $j=1,\ldots,p_t$
\begin{equation*}
    \begin{aligned}
        & (\mdot{v}_{h_t}^{p_t}, \mathcal{L}_j^{(n)})_{L^2(t_n,t_{n+1})} + \mu (u_{h_t}^{p_t}, \mathcal{L}_j^{(n)})_{L^2(t_n,t_{n+1})} + (\Int g(u_{h_t}^{p_t}), \mathcal{L}_j^{(n)})_{L^2(t_n,t_{n+1})} 
        \\ & \quad\quad\quad = \mdot{v}_{h_t}^{p_t} (x_j^{(n)}) \omega_j^{(n)} + \mu u_{h_t}^{p_t}(x_j^{(n)}) \omega_j^{(n)} + \Int g(u_{h_t}^{p_t})(x_j^{(n)}) w_j^{(n)}
        \\ & \quad\quad\quad = \mdot{v}_{h_t}^{p_t} (x_j^{(n)}) \omega_j^{(n)} + \mu u_{h_t}^{p_t}(x_j^{(n)}) \omega_j^{(n)} + g(u_{h_t}^{p_t})(x_j^{(n)}) w_j^{(n)}
	\end{aligned}
\end{equation*}
and
\begin{equation*}
		(\Int f,\mathcal{L}_j^{(n)})_{L^2(t_n,t_{n+1})} = \Int f(x_j^{(n)}) w_j^{(n)} = f(x_j^{(n)}) w_j^{(n)}
\end{equation*}
from which we get
\begin{equation} \label{eq:41}
	\begin{aligned}
	\mdot{v}_{h_t}^{p_t} (x_j^{(n)}) + \mu u_{h_t}^{p_t}(x_j^{(n)}) + g(u_{h_t}^{p_t})(x_j^{(n)}) & = f(x_j^{(n)}), \qquad 1\le j\le p_t. 
	\end{aligned}
\end{equation}
Now, consider applying a $p_t$-stage Gauss--Legendre Runge--Kutta method to the initial value problem 
\begin{equation*}
\begin{cases} 
    \mdot u = v &  \text{in~} [0,T],
    \\ 
    \mdot v = - \mu u - g(u) + f & \text{in~} [0,T],
    \\  u(0) = u_0, \quad v(0) = v_0.
\end{cases}
\end{equation*}
In the notation of \eqref{eq:38}, we have $\boldsymbol{y}(t) = (u(t), v(t))$ and $\boldsymbol{F}(t,\boldsymbol{y}(t)) = (v(t), - \mu u(t) + f(t))$. Then, from~\eqref{eq:38}, 
we deduce, for all $n \ge 1$, that the Gauss--Legendre Runge--Kutta approximation $\boldsymbol{q}_n = ({u_{h_t}^{p_t}}_{|_{[t_{n-1},t_n]}}, {v_{h_t}^{p_t}}_{|_{[t_{n-1},t_n]}})$ solves 
\begin{equation} \label{eq:42}
\begin{aligned}
    \boldsymbol{q}_n(t_{n-1}) & = (u_{h_t}^{p_t}(t_{n-1}), v_{h_t}^{p_t}(t_{n-1})), \\
    \mdot{\boldsymbol{q}}_n(x_j^{(n)}) & = \bigl(v_{h_t}^{p_t}(x_j^{(n)}), -\mu u_{h_t}^{p_t}(x_j^{(n)}) -g(u_{h_t}^{p_t})(x_j^{(n)}) + f(x_j^{(n)})\bigr), \quad j = 1,\ldots,p_t,
\end{aligned}
\end{equation}
since $x_j^{(n)} = t_{n-1} + c_j h_t^{(n)}$. Comparing  \eqref{eq:42} with \eqref{eq:40}--\eqref{eq:41}, we conclude that the scheme \eqref{eq:39} is exactly a Gauss--Legendre Runge--Kutta scheme with $p_t$ stages.

\section*{Funding}
\noindent
This research was funded in part by the Austrian Science Fund (FWF) \href{https://doi.org/10.55776/F65}{10.55776/F65}. MF and EZ are members of the Gruppo Nazionale Calcolo Scientifico-Istituto Nazionale di Alta Matematica (GNCS-INdAM).

\bibliography{bibliography}{}
\bibliographystyle{plain}

\end{document}